\documentclass[11pt,twoside]{amsart}

\usepackage{amsmath,amssymb
}
\usepackage{amsfonts}
\usepackage{amsthm,amscd}
\usepackage{amsmath}
\usepackage{amssymb}
\usepackage{amstext}
\usepackage{color}
\usepackage{latexsym}
\usepackage{amscd,graphics}
\usepackage{enumerate}

\usepackage{dashundergaps}

\usepackage{tikz}
\usetikzlibrary{calc}
\usetikzlibrary{decorations.pathreplacing,angles,quotes}

\usepackage{float}

\usepackage{bbm}
\usepackage{empheq}
\usepackage{mathtools}
\PassOptionsToPackage{hyphens}{url}
\usepackage{hyperref}
\hypersetup{colorlinks = false, urlcolor = blue, bookmarksopen = true}

\usepackage{geometry}
\allowdisplaybreaks

\usepackage{scalerel}
\usepackage{stackengine,wasysym}

\newcommand{\cA}{\mathcal{A}}

\newcommand{\cD}{\mathcal{D}}

\newcommand{\cH}{\mathcal{H}}
\newcommand{\cI}{\mathcal{I}}

\newcommand{\cM}{\mathcal{M}}
\newcommand{\cN}{\mathcal{N}}

\newcommand{\cP}{\mathcal{P}}
\newcommand{\cQ}{\mathcal{Q}}

\newcommand{\cU}{\mathcal{U}}
\newcommand{\cV}{\mathcal{V}}

\newcommand{\bN}{\mathbbm{N}}

\newcommand{\bR}{\mathbbm{R}}

\newcommand{\oQ}{\mathring{Q}}

\newcommand{\ttH}{\widetilde{H}}
\newcommand{\tth}{\widetilde{h}}
\newcommand{\tthtop}{\widetilde{h}_{\scriptsize{\mbox{\rm top}}}}
\newcommand{\ttP}{\widetilde{P}}
\newcommand{\ttPtop}{\widetilde{P}_{\scriptsize{\mbox{\rm top}}}}
\newcommand{\Ptop}{P_{\scriptsize{\mbox{\rm top}}}}

\newcommand{\htop}{h_{\scriptsize{\mbox{\rm top}}}}

\newcommand{\ve}{\varepsilon}

\newcommand{\diam}{\mbox{diam}}

\newcommand{\intd}{\,{\rm d}}

\DeclareMathOperator{\supp}{supp}

\newcommand{\unidom}{\overset{\scriptsize{\mbox{\rm unif}}}{\geqslant}}

\newcommand{\comments}[1]{} %Comment

\newtheorem{proposition}{Proposition}[section]
\newtheorem{lemma}[proposition]{Lemma}

\newtheorem{theorem}[proposition]{Theorem}

\newtheorem{question}[proposition]{Question}

\newtheorem{corollary}[proposition]{Corollary}

\newtheorem{definition}[proposition]{Definition}

\theoremstyle{remark}

\theoremstyle{definition}

\numberwithin{equation}{section}

\begin{document}
\title[Entropy via partitions of unity]{Entropy structures with continuous partitions of unity}

\author{J{\'e}r{\^o}me Carrand}
\address{Laboratoire Paul Painlev\'e, Universit\'e de Lille, 59655 Villeneuve-d'Ascq, France}
\email{jerome.carrand@univ-lille.fr}

\date{\today}
\begin{abstract}
Using only continuous partitions of unity, we provide equivalent definitions for the metric, topological and topological tail entropies and pressures of a continuous self-map of a compact set, as well as their conditional versions. A tail variational principle for these new definitions is proved. We extend Downarowicz's notions of candidates and entropy structures to account for almost-increasing sequences of functions arising from the new definitions. Finally, we deduce a partial answer to a question raised by Newhouse.
\end{abstract}
\thanks{ }
\maketitle

\section{Introduction}

Consider a topological dynamical system $(X,T)$, that is a compact metric set $X$ and a continuous surjective map $T : X \to X$. Intuitively, chaotic behaviours should emerge when iterating $T$. The most widely used notions to measure this phenomenon is entropy. Among the notions of entropy within dynamical systems, we distinguish two of the classical ones: a purely topological notion (known as \emph{topological entropy}) which is just a non-negative number $\htop(T)$, and a measure-theoretical one (known as \emph{metric entropy}) which is a function $h : \mu \mapsto h_\mu(T)$ over $\cM(X,T)$, the simplex of $T$-invariant measures (by which we mean the Borel probabilities) on $X$. Although these two notions are different in nature, they are related by a variational principal: the topological entropy is the supremum of the metric entropy function~\cite{goodwyn69,goodman71,misiurewicz76a}.

These two notions of entropy enjoy many equivalent definitions, usually obtained by considering limits of sequences of intermediate entropy notions defined at progressively finer scales. These intermediate entropy notions are often referred to as \emph{local} entropy. In the last forty years, a series of works has been published to study the connections between these topological and measure-theoretical notions, notably through new variational principles~\cite{BK83, LW77, DS02, BD04, K80, romagnoli03, HY07, OW93, N89, downarowicz05, LCC12}. Generalisations of entropy to operators have also been investigated~\cite{GLW86, M00, DF05}.

\medskip

To understand the role of small scales in the emergence of entropy, one can consider a sequence of local entropy functions $h_k$, taken at increasingly smaller scales, converging point-wise to the metric entropy $h$. Usually, this convergence is not uniform, and it is then important to characterize the ``type of non-uniformity" of the convergence. As the choice of the sequence $h_k$ is somewhat arbitrary, Downarowicz introduced in \cite{downarowicz05} a natural equivalence relation (called \emph{uniform equivalence}) among abstract non-decreasing sequences (called \emph{candidates}) of functions converging point-wise to $h$. Uniformly equivalent candidates converge with the same type of non-uniformity to $h$. Furthermore, there is a particularly important equivalence class which can be thought as a ``master invariant". This class and its elements are called \emph{entropy structures}. Indeed, most commonly known entropy notions can be recovered from any entropy structure by taking several combinations of limits and suprema: obviously the metric entropy and the topological entropy can be recovered, but also Misiurewicz topological tail entropy~\cite{misiurewicz76b}, as well as new invariant coming from the theory of symbolic extensions~\cite{BD04} such as the symbolic extension entropy function. Other new invariants appear: a transfinite sequence of functions (which is stationary) and its order of accumulation~\cite{BMc12}. We refer to Downarowicz's book~\cite{down11book} on entropy for a complete introduction. 

Moreover, Downarowicz proved in~\cite{downarowicz05} that most known definitions of entropy (and their associated notions of local entropy) lead to entropy structures. Among those passing definitions are the ones attributed to Bowen, Katok, Romagnoli, Brin and Katok, Ornstein and Weiss, Newhouse, all recalled in~\cite{downarowicz05}. To do so, Downarowicz introduced a new definition involving finite family of functions over $X$. The definition of Ghys, Langevin and Walczak~\cite{GLW86} involving measurable partitions of unity is not covered. 

Surprisingly, one possible exception is the classical definition involving partitions. Indeed, in \cite{downarowicz05}, Downarowicz provided an example of a refining sequence of partitions whose associated candidate fails to be an entropy structure. On the other hand, Lindenstrauss \cite{lindenstrauss99} provided a condition guaranteeing the existence of refining sequence of essential partition, for which the associated candidate is an entropy structure.

\medskip

A natural question, raised by Newhouse~\cite{downarowicz05}, is then: Does there always exist a sequence of refining partitions such that its associated candidate is an entropy structure?

\medskip

In the present paper, we prove that from any refining sequence of Borel partitions, we can construct a candidate which is an entropy structure, where the definitions of candidates and uniform equivalence are slightly weakened to admit \emph{almost} non-decreasing sequences $h_k$. This is done by investigating a suitable definition of the metric entropy involving \emph{continuous partitions of unity}, which is inspired from the Ghys--Langevin--Walczak entropy definition~\cite{GLW86}. We also derive new definitions of the topological entropy and of the topological tail entropy involving only continuous partitions of unity, as well as their pressure counterparts.

\subsection{Organization of the paper and statement of results}

Sections~\ref{sect: def and equiv}--\ref{sect: vanish diam} are dedicated to the systematic study of a new definition of metric entropy involving continuous partitions of unity. More precisely, Ghys, Langevin and Walczak~\cite{GLW86} introduced a definition of the metric entropy obtained by taking the supremum over all finite \emph{measurable} partitions of unity $\Phi$ of a suitable notion of local entropy $\tth_\mu(T,\Phi)$. In this definition, we restrict the supremum to the class of \emph{continuous} partitions of unity (Definition~\ref{def: metric entropy}), and prove that the obtained function $\mu \mapsto \tth_\mu(T)$ is still the metric entropy. We also derive a new definition of the topological entropy $\tthtop(T)$ involving only continuous partitions of unity (Definition~\ref{def: topological entropy}). We give a condition for a sequence of continuous partitions of unity to be maximizing, uniformly on the measure. 

More precisely, we prove.
\begin{theorem}\label{main thm 1}
For any topological dynamical system $(X,T)$ with finite topological entropy, it holds $\tthtop(T) = \htop(T)$, and for any $T$-invariant measure $\mu$, $\tth_{\mu}(T) = h_{\mu}(T)$.\\
Furthermore, if $\Phi_n$ is a sequence of continuous partitions of unity such that $\diam(\Phi_n) = \max \{ \diam (\supp \varphi) \mid \varphi \in \Phi_n \}$ converges to $0$ as $n$ goes to infinity, then for every $T$-invariant measure $\mu$, $\Phi_n$ is maximizing in the sense
\begin{align*}
\tth_\mu(T) = \lim_{n \to \infty} \tth_\mu(T,\Phi_n).
\end{align*}
\end{theorem}
A similar statement is obtained for the pressure of a continuous potential in Section~\ref{sect: pressure}. Theorem~\ref{main thm 1} follows from Corollary~\ref{corol: eq of defs metr and top} and Theorem~\ref{thm: vanish diam metric}.

In Section~\ref{sect: top tail and tail var pple}, we introduce a quantity $\tth^*(T)$ defined using continuous partitions of unity (Definition~\ref{def: top tail entropy V1}), and we prove that is coincides with Misiurewicz's topological tail entropy $h^*(T)$. Furthermore, we prove that for any sequence $\Psi_n$ with vanishing diameters, then the tail variational principle holds. That is

\begin{theorem}\label{main thm 2}
For any topological dynamical system $(X,T)$, then $\tth^*(T) = h^*(T)$. Furthermore, if $\Psi_k$ is a sequence of continuous partitions of unity with $\diam(\Psi_k)$ converging to zero, then
\begin{align*}
\lim_{k \to \infty} \sup_{\mu \in \cM(X,T)} h_\mu(T) - \tth_\mu(T,\Psi_k) = h^*(T).
\end{align*}
\end{theorem}
A similar result for topological tail pressure is obtained in Section~\ref{sect: top tail pressure}. Theorem~\ref{main thm 2} follows from Theorem~\ref{thm: tail var ppl} and Corollary~\ref{corollary: equiv top tail entropy}.

Section~\ref{sect: entropy structure} is devoted to the notion of entropy structures (Definition~\ref{def: entropy structure}). We derive two families of entropy structures using continuous partitions of unity. Furthermore, we extend the definition of \emph{candidates}, into \emph{weak-candidates}, in order to include in particular sequences of the form $\cH_\Psi = (\mu \mapsto \tth_\mu(T,\Psi_k))_{k \in \bN}$ with $\diam(\Psi_k)$ converging to zero which are not increasing but \emph{almost}. We extend the definition of uniform equivalence accordingly into \emph{weak uniform equivalence} (Definition~\ref{def: weak}). Finally, we prove that any weak-candidate weakly-uniformly equivalent to an entropy structure has the same transfinite sequence of functions, and the same super-envelops as any entropy structure. We show that the sequence $\cH_\Psi$ is such a weak-candidate. In particular, the symbolic extension entropy can be recovered from $\cH_\Psi$, as well as $h^*(T)$. We summarise all of the above into the following

\begin{theorem}\label{main thm 3}
For any topological dynamical system $(X,T)$, and any sequence $\Psi_k$ of continuous partitions of unity with $\diam(\Psi_k)$ converging to zero, then the sequence $\cH = (\mu \mapsto \tth_\mu(T,\Psi_k))_{k \in \bN}$ is a weak-candidate weakly-uniformly equivalent to any entropy structure. In particular, $\cH$ shares the same transfinite sequence of functions $(u_\alpha)_\alpha$ and super-envelops $E$ as any entropy structure. 
\end{theorem}
It follows by combining Theorems~\ref{thm: same repair and superenvelops} and~\ref{thm: weak entropy structure}.

Finally, in Section~\ref{sect: open}, we mention some problems still open.

\section{Topological and metric entropies}\label{sect: def and equiv}

In this section we define quantities analogous to the usual metric and topological entropies in Definitions~\ref{def: metric entropy} and~\ref{def: topological entropy}, but using continuous partitions of unity. The metric entropy function is formally the same the one introduced by Ghys, Langevin and Walczak~\cite{GLW86} initially for measurable partitions of unity. We then prove that these quantities coincide with their classical counterparts in Corollary~\ref{corol: eq of defs metr and top}. Finally, we prove that the local metric entropy function is upper semi-continuous in Proposition~\ref{prop: upper semi-continuity}.

\subsection{Metric entropy via partitions of unity}

Given partitions over a set and a self-map over this same set, we can define the join of partitions as well as the pullback of a partition. These definitions naturally extend to partitions of unity as follows.  

\begin{definition}
A finite collection of function $\Phi = \{ \varphi_i : X \to [0,1] \mid 1 \leqslant i \leqslant n \}$ such that $\sum_{i=1}^n \varphi_i \equiv 1$ is called a partition of unity (over $X$). If all functions in $\Phi$ are measurable (resp. continuous), the partition of unity is said measurable (resp. continuous). \\
If $\Phi$ and $\Psi$ are two partition of unity, their join is defined by
\begin{align*}
\Phi \vee \Psi \coloneqq \{ \varphi \psi \mid \varphi \in \Phi, \, \psi \in \Psi \}.
\end{align*}
A self map $T : X \to X$ acts on partitions of unity as follows: if $\Phi$ is a partition of unity, then
\begin{align*}
T \Phi \coloneqq \{ \varphi \circ T \mid \varphi \in \Phi \},
\end{align*}
is also a partition of unity. \\
We also define for any $n \geqslant 1$, the partition of unity
\begin{align*}
\Phi_0^{n-1} \coloneqq \bigvee_{i=0}^{n-1} T^i \Phi.
\end{align*}
\end{definition}

In \cite{GLW86}, Ghys, Langevin and Walczak defined the metric entropy of a continuous self-map using measurable partitions of unity, and proved that this new quantity coincides with the classical one. We recall their definition here, restricted to the case of continuous partitions of unity,

\begin{definition}[metric entropy]\label{def: metric entropy}
Given $\mu \in \cM(X,T)$ and a finite and measurable partition of unity $\Phi$, the static entropy of $(T,\mu)$ with respect to $\Phi$ is given by
\begin{align*}
\ttH_\mu(\Phi) \coloneqq \sum_{\varphi \in \Phi} - \mu(\varphi) \log \mu(\varphi) + \mu(\varphi \log \varphi).
\end{align*}
The (local) entropy of $(T,\mu)$ with respect to $\Phi$ is defined by
\begin{align*}
\tth_{\mu}(T,\Phi) \coloneqq \lim_{n \to \infty} \frac{1}{n} \ttH_{\mu}(\Phi_0^{n-1}).
\end{align*}
The entropy of $(T,\mu)$ is defined by
\begin{align*}
\tth_{\mu}(T) \coloneqq \sup \{ \tth_{\mu}(T,\Phi) \mid \Phi \text{ continuous partition of unity} \}.
\end{align*}
\end{definition}

Notice that for a measurable partition $\cA$ of $X$, and $\mathbbm{1}_\cA = \{ \mathbbm{1}_A \mid A \in \cA \}$, we recover the classical definitions of the static entropy $H_\mu(\cA) = \ttH_\mu(\mathbbm{1}_\cA)$, and of the local entropy $h_\mu(T,\cA) = \tth_\mu(T,\mathbbm{1}_\cA)$.

In the rest of the paper, all quantities (entropies, pressures, etc.) defined using partitions of unity will be denoted with a \emph{tilde} in order to distinguish them from their classical counterparts.

In order to prove that such entropy is well defined and several other properties, Ghys, Langevin and Walczak \cite{GLW86} also defined a conditional metric entropy using (measurable) partitions of unity. We recall here their definition.

\begin{definition}
For a given probability measure $\mu$ and a measurable function $ \psi \geqslant 0$ such that $\mu(\psi) >0$, we define the conditional measure $\mu_\psi$ of $\mu$ with respect to $\psi$ by $\mu_\psi(\varphi) \coloneqq \mu(\varphi \psi) / \mu(\psi)$. Notice that $\mu_\psi$ is a probability measure.

For $\Psi$ and $\Phi$ two (measurable) partitions of unity, denote the conditional static entropy of $(T,\mu)$ with respect to $\Phi$ knowing $\Psi$
\begin{align*}
 \ttH_\mu(\Phi | \Psi) \coloneqq \sum_{\psi \in \Psi} \mu(\psi) \ttH_{\mu_\psi}(\Phi),
\end{align*}
where, as a convention, we set $\mu(\psi) \ttH_{\mu_\psi}(\Phi) = 0$ whenever $\mu(\psi) =0$.

Define the conditional entropy of $(T,\mu)$ with respect to $\Phi$ knowing $\Psi$ by
\begin{align*}
\tth_\mu(T,\Phi | \Psi) \coloneqq \lim_{n \to \infty} \frac{1}{n} \ttH_\mu(\Phi_0^{n-1} | \Psi_0^{n-1})
\end{align*}
\end{definition}

The existence of the limits in the definitions of entropy and conditional entropy follows from the next two lemmas. They are part of \cite{GLW86}, with different notations from here. We repeat their proof for completeness.

\begin{lemma}\label{lemma: basic prop static}
Given partitions of unity $\Phi$ and $\Psi$, and a $T$-invariant measure $\mu$, the following hold

\noindent
\textit{i)} $\ttH_{\mu}(\Phi \vee \Psi) = \ttH_{\mu}(\Psi) + \ttH_{\mu}(\Phi | \Psi)$,

\noindent
\textit{ii)} $\ttH_{\mu}(\Phi_1 \vee \Phi_2 | \Psi) = \ttH_{\mu}(\Phi_1 | \Psi) + \ttH_{\mu}(\Phi_2 | \Phi_1 \vee \Psi)$,

\noindent
\textit{iii)} $\ttH_{\mu}(\Phi | \Psi_1 \vee \Psi_2) \leqslant \ttH_{\mu}(\Phi | \Psi_1)$,

\noindent
\textit{iv)} $\ttH_{\mu}(\Phi) \geqslant 0$ and $\ttH_{\mu}(\Phi | \Psi) \geqslant 0$.
\end{lemma}

Contrary to the case of partitions where $\cA \vee \cA = \cA$, for partitions of unity $\Phi \vee \Phi \neq \Phi$ in general. As a consequence, $\ttH_\mu(\Phi | \Phi) \neq 0$ in general. The $\Phi$ for which equality holds have been characterised by Kaminski and Sam Lazaro \cite{KSL00}: each $\varphi \in \Phi$ is a multiple of some indicator function.

\begin{proof}
For point $i)$, we expand the definition as follows
\begin{align*}
\ttH_{\mu}&(\Phi | \Psi) = \sum_{\psi \in \Psi} \mu(\psi) \ttH_{\mu_\psi} = \sum_{\psi \in \Psi} \mu(\psi) \sum_{\varphi \in \Phi} - \frac{\mu(\varphi \psi)}{\mu(\psi)} \log \frac{\mu(\varphi \psi)}{\mu(\psi)} + \frac{\mu(\varphi \psi \log \varphi)}{\mu(\psi)} \\
&=\sum_{\psi \in \Psi} \sum_{\varphi \in \Phi} -\mu(\varphi \psi) \log \mu(\varphi \psi) + \mu(\varphi \psi) \log \mu(\psi) + \mu(\varphi \psi \log \varphi) + \mu(\varphi \psi \log \psi) - \mu(\varphi \psi \log \psi) \\
&=\sum_{\varphi \psi \in \Phi \vee \Psi} -\mu(\varphi \psi) \log \mu(\varphi \psi) + \mu(\varphi \psi \log \varphi \psi) - \sum_{\psi \in \Psi} -\mu(\psi) \log \mu(\psi) + \mu(\psi \log \psi) \\
&= \ttH_{\mu}(\Phi \vee \Psi) - \ttH_{\mu}(\Psi). 
\end{align*}

In order to prove $ii)$, we use $i)$ three times:
\begin{align*}
\ttH_{\mu}(\Phi_1 \vee \Phi_2 | \Psi) &= \ttH_{\mu}(\Psi_1 \vee \Phi_2 \vee \Psi) - \ttH_{\mu}(\Psi) + \ttH_{\mu}(\Phi_1 \vee \Psi) - \ttH_{\mu}(\Phi_1 \vee \Psi) \\
&= \ttH_{\mu}(\Phi_2 | \Phi_1 \vee \Psi) + \ttH_{\mu}(\Phi_1 | \Psi).
\end{align*}

For point $iii)$, we use that $\eta(t) = t \log t$, $ t \in [0,1]$ is convex. Indeed, for any $\varphi \in \Phi$ and $\psi_1 \in \Psi_1$,
\begin{align*}
\frac{\mu( \varphi \psi_1)}{\mu(\psi_1)} \log \frac{\mu( \varphi \psi_1)}{\mu(\psi_1)} 
&= \eta \left( \frac{\mu( \varphi \psi_1)}{\mu(\psi_1)} \right) 
= \eta \left( \sum_{\psi_2 \in \Psi_2} \frac{\mu(\psi_1 \psi_2)}{\mu(\psi_1)} \frac{\mu(\varphi \psi_1 \psi_2)}{\mu(\psi_1 \psi_2)} \right) \\
&\leqslant \sum_{\psi_2 \in \Psi_2} \frac{\mu(\psi_1 \psi_2)}{\mu(\psi_1)} \eta \left( \frac{\mu(\varphi \psi_1 \psi_2)}{\mu(\psi_1 \psi_2)} \right).
\end{align*}
Thus,
\begin{align*}
- \ttH_{\mu}(\Phi |\Psi_1) 
&= \sum_{\psi_1 \in \Psi_1} \mu(\psi_1) \left( \sum_{\varphi \in \Phi} \eta \left( \frac{\mu( \varphi \psi_1)}{\mu(\psi_1)} \right) - \frac{\mu( \varphi \psi_1 \log \varphi)}{\mu(\psi_1)} \right) \\
&\leqslant \sum_{\psi_1 \in \Psi_1} \sum_{\varphi \in \Phi} \sum_{\psi_2 \in \Psi_2} \mu(\psi_1 \psi_2) \left( \eta(\mu_{\psi_1 \psi_2}(\varphi)) - \mu_{\psi_1 \psi_2}(\varphi \log \varphi) \right) \\
&\leqslant - \ttH_{\mu}(\Phi | \Psi_1 \vee \Psi_2).
\end{align*}

Point $iv)$ is an application of the Jensen inequality. Indeed, since $\eta$ is convex, for any probability measure $\mu$, and function $\varphi$ taking values in $[0,1]$, we have $\eta(\mu(\varphi)) \leqslant \mu(\eta \circ \varphi)$. In particular, for any partitions of unity $\Phi$ and $\Psi$, $\ttH_{\mu_{\psi}}(\Phi) \geqslant 0$ for each $\psi \in \Psi$, and therefore $\ttH_\mu(\Phi |\Psi) \geqslant 0$. Taking $\Psi = \{ 1 \}$, yields $\ttH_\mu(\Phi) \geqslant 0$.  
\end{proof}

\begin{lemma}
The sequences $\left( \ttH_{\mu}(\Phi_0^{n-1}) \right)_n$ and $\left( \ttH_{\mu}(\Phi_0^{n-1} | \Psi_0^{n-1}) \right)_n$ are sub-additive. In particular, the metric entropy and conditional entropy are well defined. 
\end{lemma}

\begin{proof}
By using Lemma~\ref{lemma: basic prop static}, we get that for any $n$ and $m$ larger than $1$,
\begin{align*}
\ttH_{\mu}(\Phi_0^{n+m-1} | \Psi_0^{n+m-1}) 
&= \ttH_{\mu}(\Phi_0^{n-1} \vee T^n \Phi_0^{m-1} | \Psi_0^{n-1} \vee T^n \Psi_0^{m-1}) \\
&= \ttH_{\mu}(\Phi_0^{n-1} | \Psi_0^{n-1} \vee T^n \Psi_0^{m-1}) + \ttH_{\mu}( T^n \Phi_0^{m-1} | \Phi_0^{n-1} \vee \Psi_0^{n-1} \vee T^n \Psi_0^{m-1}) \\
&\leqslant \ttH_{\mu}(\Phi_0^{n-1} | \Psi_0^{n-1}) + \ttH_{\mu}( T^n \Phi_0^{m-1} | T^n \Psi_0^{m-1}).
\end{align*}
Now, by $T$-invariance of $\mu$, we get that 
\begin{align*}
\ttH_{\mu}( T^n \Phi_0^{m-1} | T^n \Psi_0^{m-1}) = \ttH_{\mu}( \Phi_0^{m-1} | \Psi_0^{m-1}),
\end{align*}
thus the sub-additivity of $\left( \ttH_{\mu}(\Phi_0^{n-1} | \Psi_0^{n-1}) \right)_n$. 

Notice that in the special case $\Psi = \{ 1 \}$, we get $\ttH_{\mu}(\Phi_0^{n-1} | \Psi_0^{n-1}) = \ttH_{\mu}(\Phi_0^{n-1} )$ for every $n$. Hence $\left( \ttH_{\mu}(\Phi_0^{n-1}) \right)_n$ is also sub-additive.
\end{proof}

\subsection{Topological entropy via continuous partitions of unity}

We introduce now a counterpart of the topological entropy defined using continuous partition of unity. The next definition is inspired from the work of Langevin and Walczak \cite{LW94}.

\begin{definition}[topological entropy]\label{def: topological entropy}
Given a finite partition of unity $\Phi$, its topological static entropy is given by
\begin{align*}
\ttH(\Phi) \coloneqq \sum_{\varphi \in \Phi} \sup \varphi.
\end{align*}
The (local) topological entropy of $T$ with respect to $\Phi$ is defined by
\begin{align*}
\tth(T,\Phi) \coloneqq \lim_{n \to \infty} \frac{1}{n} \log \ttH(\Phi_0^{n-1}).
\end{align*}
The topological entropy of $T$ is defined by
\begin{align*}
\tthtop(T) \coloneqq \sup \{ \tth(T,\Phi) \mid \Phi \text{ continuous partition of unity} \}.
\end{align*}
\end{definition}

As in the case of metric entropy, the existence of the limit in the definition is ensured by the following lemma. 

\begin{lemma}
The sequence $\left( \log \ttH(\Phi_0^{n-1}) \right)_n$ is sub-additive. In particular the topological entropy is well defined.
\end{lemma}

\begin{proof}
For any $n$ and $m$ larger than $1$, we have that
\begin{align*}
\ttH(\Phi_0^{n+m-1}) &= \sum_{\varphi_1 \varphi_2 \circ T^n \in \Phi_0^{n-1} \vee T^n \Phi_0^{m-1}} \sup \varphi_1 \varphi_2 \circ T^n \\
&\leqslant \sum_{\varphi_1 \in \Phi_0^{n-1}} \sup \varphi_1  \sum_{\varphi_2 \in \Phi_0^{m-1}} \sup \varphi_2 \circ T^n = \ttH(\Phi_0^{n-1}) \ttH(\Phi_0^{m-1}).
\end{align*}
Taking logarithm of the above yields the claim.
\end{proof}

In this context, we can prove an analogue of Goodwyn's theorem \cite{goodwyn69}. This relies on the classical lemma we omit the proof of (see \cite[Lemma~9.9]{walters82book}).
\begin{lemma}\label{lemma: classical}
Let $a_1, \ldots , a_n$, and $p_1, \ldots, p_n$ be real numbers such that $p_i \geqslant 0$ and $\sum_{i=1}^n p_i = 1$, then
\begin{align*}
\sum_{i=1}^n p_i( a_i - \log p_i) \leqslant \log \sum_{i=1}^n e^{a_i}.
\end{align*}
Furthermore, equality holds if and only if $p_i = e^{a_i}/M$ for every $i$, where $M = \sum_{i=1}^n e^{a_i}$.
\end{lemma}

\begin{proposition}\label{prop: thmu leq thtop}
For every $T$-invariant measure $\mu$ and continuous partition of unity $\Phi$, 
\begin{align*}
\ttH_{\mu}(\Phi) \leqslant \log \ttH(\Phi).
\end{align*}
As a consequence, $\tth_{\mu}(T,\Phi) \leqslant \tthtop(T,\Phi)$ and $\tth_{\mu}(T) \leqslant \tthtop(T)$.
\end{proposition}

\begin{proof}
Using Lemma~\ref{lemma: classical} with $p_i = \mu(\varphi_i)$,
\begin{align*}
\log \ttH(\Phi) &= \log \sum_{\varphi \in \Phi} e^{\log \sup \varphi} 
\geqslant \sum_{\varphi \in \Phi} \mu(\varphi) \left( -\log \mu(\varphi) + \log \sup \varphi \right) \\
&\geqslant \sum_{\varphi \in \Phi} - \mu(\varphi) \log \mu (\varphi) + \mu (\varphi \log \varphi) = \ttH_{\mu}(\Phi).
\end{align*}
Replacing $\Phi$ by $\Phi_0^{n-1}$ and taking limit yield $\tth_\mu(T,\Phi) \leqslant \tthtop(T,\Phi)$. Taking supremum in $\Phi$ finishes the proof.
\end{proof}

\subsection{Equivalence of definitions}

In this part, we prove that the ``tilded" versions of entropy coincide with their classical counterparts. For the classical definitions of the metric and topological entropies, see, for example, \cite[Definitions~4.9 and 7.5]{walters82book}. We split these results into several inequalities. This first one seems to be new.

\begin{proposition}\label{prop: thmu geq hmu}
For every $T$-invariant measure $\mu$, $\tth_{\mu}(T) \geqslant h_{\mu}(T)$.
\end{proposition}

\begin{proof}
Fix $\mu$ an invariant measure and $\ve_0 > 0$. Let $\cP = \{ P_1 , \ldots , P_k \}$ be a finite partition of $X$ such that 
\begin{align*}
h_{\mu}(T, \cP) \geqslant h_\mu(T) - \ve_0.
\end{align*}
Without loss of generality, we can assume that $\mu(P_i)>0$ for every $1 \leqslant i \leqslant k$.

We proceed in two steps. First we construct a partition $\cQ$ from $\cP$ whose atoms have non-empty interior and zero measure boundaries. We then construct a continuous partition of unity $\Phi$ from $\cQ$.

Let $\ve > 0$. Since $\mu$ is regular, for each $1 \leqslant i \leqslant k$ there exists a closed set $F_i \subset P_i$ such that $\mu ( F_i \triangle P_i) < \ve$. Furthermore, we can choose an open set $O_i \supset F_i$ such that $\mu ( O_i \triangle F_i ) < \ve$ and $\mu ( \partial O_i) = 0$. Such open set can be obtain by taking a $\delta$-neighbourhood of $F_i$.

Let $Q_0 \coloneqq \bigcup_{1 \leqslant i < j \leqslant k} O_i \cap O_j$ and $Q_i = O_i \smallsetminus Q_0$ for $1 \leqslant i \leqslant k$. Finally, set $Q_{-1} = X \smallsetminus \bigcup_{0 \leqslant i \leqslant k} Q_i$. Denote $\cQ \coloneqq \{ Q_{-1}, \ldots , Q_k \}$ the obtained partition of $X$.

We can estimate the measures of $Q_0$ and $Q_{-1}$ as follows.
\begin{align*}
\mu(Q_0) &\leqslant \sum_{1 \leqslant i < j \leqslant k} \mu (O_i \cap O_j) = \sum_{1 \leqslant i < j \leqslant k} \mu ( (F_i \sqcup (F_i \triangle O_i)) \cap (F_j \sqcup (F_j \triangle O_j )) ) \\
&\leqslant 2\sum_{i=1}^{k} \mu (F_i \triangle O_i) + \sum_{i<j} \ve \leqslant 3 \frac{k(k+1)}{2} \ve , \\
\mu(Q_{-1}) &= \sum_{i=1}^k \mu( P_i \smallsetminus \bigcup_{j=1}^k O_j) \leqslant \sum_{i=1}^k \mu( P_i \smallsetminus \bigsqcup_{j=1}^k F_j) = \sum_{i=1}^k \mu( P_i \smallsetminus F_i) \leqslant k \ve.
\end{align*}
Furthermore, $\mu( \partial Q_i ) = 0$ for every $-1 \leqslant i \leqslant k$.

We can now estimate from below the entropy $h_\mu(T,\cQ)$. For this, we use that 
\begin{align*}
h_\mu(T,\cP) - h_\mu(T,\cQ) \leqslant H_\mu(\cP | \cQ) = \sum_{i=-1}^{k} \mu(Q_i) H_{\mu_{Q_i}}(\cP).
\end{align*}

We split the sum into the three contributions corresponding to $1 \leqslant i \leqslant k$, to $i=0$ and to $i=-1$.

For $1 \leqslant i \leqslant k$, since
\begin{align*}
\mu(P_i \cap Q_i) &= \mu ( (P_i \cap O_i) \smallsetminus Q_0) = \mu ( P_i \cap (F_i \sqcup (O_i \smallsetminus F_i)) \smallsetminus Q_0) \\
&= \mu(F_i \smallsetminus Q_0) + \mu(P_i \cap(O_i \smallsetminus F_i) \smallsetminus Q_0) \geqslant \mu(F_i \smallsetminus Q_0) \geqslant \mu(F_i) - 3 \frac{k(k+1)}{2} \ve, \\
\mu(Q_i) &= \mu(O_i \smallsetminus Q_0) = \mu (F_i \smallsetminus Q_0) + \mu ((O_i \smallsetminus F_i) \smallsetminus Q_0) \\
&\leqslant \mu (F_i \smallsetminus O_i) + \mu (O_i \triangle F_i) \leqslant \mu(F_i) + \ve,
\end{align*}
and from the uniform continuity of $\eta(t) = -t\log t$ on $[0,1]$, we can choose $\ve$ small enough (to make $\mu_{Q_i}(P_i)$ close enough to $\mu(P_i)$), so that
\begin{align*}
\sum_{i=1}^k \mu(Q_i) (-\mu_{Q_i}(P_i) \log \mu_{Q_i}(P_i)) < \frac{\ve_0}{4}.
\end{align*}
For $1 \leqslant i,j \leqslant k$, $i \neq j$,
\begin{align*}
\mu(P_j \cap Q_i) &= \mu(P_i \cap (O_i \smallsetminus Q_0)) = \mu( P_j \cap (F_i \sqcup (O_i \smallsetminus F_i) \smallsetminus Q_0) \\
&= \mu((P_j \cap F_i) \smallsetminus Q_0) + \mu ( P_j \cap (O_i\smallsetminus F_i) \smallsetminus Q_0) \leqslant \mu(O_i \smallsetminus F_i) \leqslant \ve, \\
\mu(Q_i) &= \mu ( O_i \smallsetminus Q_0) \geqslant \mu(O_i) - \mu(Q_0) \geqslant \mu(P_i) - \left( 2 + 3 \frac{k(k+1)}{2} \right) \ve.
\end{align*}
We can thus take $\ve$ small enough (to make $\mu_{Q_i}(P_j)$ close enough to $0$), so that
\begin{align*}
\sum_{i=1}^k \mu(Q_i) \sum_{j \neq i} -\mu_{Q_i}(P_j) \log \mu_{Q_i}(P_j) < \frac{\ve_0}{4}.
\end{align*}
In particular, the contribution of the $Q_i$, $1 \leqslant i \leqslant k$ in $H_\mu(\cP | \cQ)$ is bounded by $\ve_0 / 2$.

The contribution of $Q_0$ in $H_\mu(\cP | \cQ)$ is
\begin{align*}
\mu(Q_0) \sum_{j=1}^k -\mu_{Q_0}(P_j) \log \mu_{Q_0}(P_j) \leqslant \mu(Q_0) \log k \leqslant 3 \ve \frac{k(k+1)}{2} \log k,
\end{align*}
and the contribution of $Q_{-1}$ in $H_\mu(\cP | \cQ)$ is
\begin{align*}
\mu(Q_{-1}) \sum_{j=1}^k -\mu_{Q_{-1}}(P_j) \log \mu_{Q_{-1}}(P_j) \leqslant \mu(Q_{-1}) \log k \leqslant k \ve \log k.
\end{align*}
We can take $\ve$ small enough so that both contribution of $Q_0$ and $Q_{-1}$ are bounded by $\ve_0/4$. Therefore $H_\mu(\cP | \cQ) \leqslant \ve_0$ and thus
\begin{align*}
h_\mu(T,\cQ) \geqslant h_\mu(T) - 2 \ve_0.
\end{align*}

Let $\delta > 0$ and consider the open cover of $X$ 
\begin{align*}
\alpha \coloneqq \{ \cN_{\delta}(\partial \cQ), \oQ_{-1}, Q_0, \oQ_1, \ldots , \oQ_k \},
\end{align*}
where $\cN_\delta(A)$ is a $\delta$-neighbourhood of a set $A$. Let $\Phi = \{ \varphi_{-2}, \ldots, \varphi_k \}$ be a continuous partition of unity subordinated to $\alpha$. We now compare $\tth_\mu(T,\Phi)$ to $h_\mu(T,\cQ)$. 
Since $h_\mu(T,\cQ) = \tth_\mu(T, \mathbbm{1}_{\cQ})$, we get that
\begin{align*}
h_\mu(T,\cQ) - \tth_\mu(T, \Phi) \leqslant \ttH_\mu(\mathbbm{1}_{\cQ} | \Phi) = \sum_{i=-2}^k \mu(\varphi_i) \sum_{j=-1}^{k} - \mu_{\varphi_i}(Q_j) \log \mu_{\varphi_i}(Q_j).
\end{align*}
We split the sum into the contributions of $\varphi_{-2}$ and of $\varphi_i$, $i \geqslant -1$.

By definition of $\Phi$, for $i \geqslant -1$, we have that $\supp \varphi_i \subset \oQ_i$. But since the sets $\oQ_i$ are disjoint, $\mu_{\varphi_i}(Q_j)$ is $0$ if $i\neq j$, and is $1$ is $i=j$. In particular, the contribution of the $\varphi_i$, $i \geqslant -1$ in $\ttH_\mu(\mathbbm{1}_{\cQ} | \Phi)$ is $0$. Therefore,
\begin{align*}
\ttH_\mu(\mathbbm{1}_{\cQ} | \Phi) = \mu(\varphi_{-2}) \ttH_{\mu_{\varphi_{-2}}}(\mathbbm{1}_\cQ) \leqslant \mu(\cN_{\delta}(\partial \cQ)) \log(k+2).
\end{align*}
Since $\mu(\partial \cQ)=0$, by taking $\delta$ small enough, we can assume that $\ttH_\mu(\mathbbm{1}_{\cQ} | \Phi) \leqslant \ve_0$.
Therefore,
\begin{align*}
\tth_{\mu}(T,\Phi) \geqslant h_{\mu} (T,\cQ) - \ve_0 \geqslant h_{\mu}(T) - 3 \ve_0.
\end{align*}
Since $\ve_0 > 0$ is arbitrary, $\tth_\mu(T) \geqslant h_\mu(T)$.
\end{proof}

The converse inequality comes from \cite{GLW86} (further detailed in \cite{KSL00}). For completeness, we repeat the argument.
\begin{proposition}\label{prop: thmu leq hmu}
For every $T$-invariant measure $\mu$, $\tth_{\mu}(T) \leqslant h_{\mu}(T)$.
\end{proposition}

\begin{proof}
Let $\Phi$ be a continuous partition of unity. Let $\varepsilon > 0$. Define $k = \# \Phi$.

By uniform continuity of $\eta(t) = t \log t$ on $[0,1]$, let $\delta >0$ be such that for all $x,y \in [0,1]$, if $|x-y| < \delta$ then $|\eta(x) - \eta(y)| < \ve/2k$.

Let $\cP$ be the partition of $[0,1]$ given by the intervals $\left[ \frac{i}{r}, \frac{i+1}{r} \right)$ and $\left[\frac{r-1}{r},1 \right]$, for $0 \leqslant i \leqslant r-2$, where $r$ is large enough so that $1/r < \delta$.

Define the partition of $X$ 
\begin{align*}
\cA = \bigvee_{\varphi \in \Phi} \varphi^{-1} \cP,
\end{align*}
and the associated partition of unity $\mathbbm{1}_{\cA} = \{ \mathbbm{1}_A \mid A \in \cA \}$. We now show that $\ttH_{\mu}(\Phi | \mathbbm{1}_{\cA}) < \ve$.

To this end, we show that for all $A \in \cA$ and $\varphi \in \Phi$
\[ \mu_A(\varphi) \log \mu_A(\varphi) + \mu_A(\varphi \log \varphi) < \ve/k  . \]

Fix $A \in \cA$ and $\varphi \in \Phi$. Let $1 \leqslant i < r$ be such that $\frac{i}{r} \leqslant \varphi \leqslant \frac{i+1}{r}$ on $A$, that is \[ 0 \leqslant \varphi - \frac{i}{r} \leqslant \frac{1}{r} < \delta, \text{ on $A$.} \]
From this, we get on the one hand, by integrating with respect to $\mu_A$, $0 \leqslant \mu_A(\varphi) - \frac{i}{r} < \delta$, and thus
\begin{align*}
\left| \eta(\mu_A(\varphi)) - \eta\left( \frac{i}{r} \right) \right| < \frac{\ve}{2k}.
\end{align*}
On the other hand, we get $|\eta(\varphi) - \eta(i/r)| \leqslant \ve/2k$, and thus, by integration
\begin{align*}
\left| \mu_A(\eta(\varphi)) - \eta \left( \frac{i}{r} \right) \right| \leqslant \frac{\ve}{2k}.
\end{align*}
Hence $| \eta(\mu_A(\varphi)) - \mu_A(\eta(\varphi))| < \ve/k$. Summing over $\varphi \in \Phi$ yields $\ttH_{\mu}(\Phi | \mathbbm{1}_{\cA}) < \ve$.

Now, we get that
\begin{align*}
\ttH_{\mu}(\Phi_0^{n-1}) \leqslant \ttH_{\mu}(\Phi_0^{n-1} \vee (\mathbbm{1}_{\cA})_0^{n-1}) = H_{\mu}(\cA_0^{n-1}) + \ttH_{\mu}(\Phi_0^{n-1} |  (\mathbbm{1}_{\cA})_0^{n-1}).
\end{align*}
However, 
\begin{align*}
\ttH_{\mu}(\Phi_0^{n-1} |  (\mathbbm{1}_{\cA})_0^{n-1}) \leqslant \sum_{i=0}^{n-1} \ttH_{\mu}(T^i \Phi |  (\mathbbm{1}_{\cA})_0^{n-1}) \leqslant \sum_{i=0}^{n-1} \ttH_{\mu}(T^i \Phi |  T^i \mathbbm{1}_{\cA}) = n \ttH_{\mu}(\Phi | \mathbbm{1}_\cA) \leqslant n \ve.
\end{align*}
Hence $\tth_{\mu}(T,\Phi) \leqslant h_{\mu}(T,\cA) + \ve \leqslant h_\mu(T) + \ve$, for every $\ve >0$. Thus $\tth_{\mu}(T,\Phi) \leqslant h_\mu(T)$ for every continuous partition of unity $\Phi$. Therefore, $\tth_{\mu}(T) \leqslant h_\mu(T)$.
\end{proof}
For latter use, notice that in the above proof, the bound $\ttH_\mu(\Phi | \mathbbm{1}_\cA) < \ve$ is uniform in $\mu \in \cM(X,T)$. We now turn to the topological entropies.

\begin{proposition}\label{prop: thtop leq htop}
For any continuous self-map $T$ of a compact metric set $X$, it holds $\tthtop(T) \leqslant \htop(T)$.
\end{proposition}

The proof of the above proposition uses a very similar strategy as the one of Proposition~\ref{prop: thmu leq hmu}. In order to apply this strategy, we need first to introduce a notion of topological conditional entropy with respect to families of functions.

\begin{definition}\label{def: local cond top entropy V1}
Given two finite families of functions $\Phi$ and $\Psi$, we define
\begin{align*}
\ttH (\Phi | \psi ) &\coloneqq \sum_{\varphi \in \Phi} \sup \{ \varphi(x) \mid x \in \supp \psi \}, \quad \psi \in \Psi, 
\end{align*}
and,
\begin{align*}
\ttH (\Phi | \Psi) &\coloneqq \max \{ \ttH(\Phi | \psi) \mid \psi \in \Psi \}.
\end{align*}
\end{definition}

\begin{proof}
Let $\Phi$ be a continuous partition of unity and $\ve > 0$. Denote $k = \# \Phi$.

Let $\cI$ be a finite open cover of $[0,1]$ of diameter less than $\ve / 2k$. Denote $\cU$ the open cover of $X$ given by
\begin{align*}
\cU \coloneqq \bigvee_{\varphi \in \Phi} \varphi^{-1} \cI.
\end{align*}
In particular, we get that for each $\varphi \in \Phi$ and $U \in \cU$ 
\begin{align*}
\sup_{U} \varphi - \inf_{U} \varphi \leqslant \frac{\ve}{2k}.
\end{align*}
For each $n \geqslant 1$, let $\cV_n \subset \cU^n$ be a sub-cover of minimal cardinality. Therefore
\begin{align*}
\ttH (\Phi_0^{n-1}) 
&= \sum_{\varphi \in \Phi_0^{n-1}} \sup \varphi 
\leqslant \sum_{V \in \cV_n} \sum_{\varphi \in \Phi_0^{n-1}} \sup \varphi \mathbbm{1}_V 
= \sum_{V \in \cV_n} \ttH(\Phi_0^{n-1} | \mathbbm{1}_V) \\
&\leqslant \# \cV_n \, \ttH(\Phi_0^{n-1} | \mathbbm{1}_{\cV_n}) 
\leqslant \# \cV_n \, \ttH(\Phi_0^{n-1} | (\mathbbm{1}_{\cU})_0^{n-1} ).
\end{align*}
Now, let $U \in \cU_0^{n-1}$ be such that $\ttH(\Phi_0^{n-1} | (\mathbbm{1}_{\cU})_0^{n-1} ) = \ttH(\Phi_0^{n-1} | \mathbbm{1}_{U} )$. Hence
\begin{align*}
\ttH(\Phi_0^{n-1} | (\mathbbm{1}_{\cU})_0^{n-1} ) \leqslant \prod_{i=0}^{n-1} \ttH(T^i \Phi | \mathbbm{1}_U) \leqslant \prod_{i=0}^{n-1} \ttH(T^i \Phi | (\mathbbm{1}_{U})_0^{n-1} ).
\end{align*}

Fix some $0 \leqslant i  < n$ and let $V \in \bigcap_{k=0}^{n-1} T^{-k} U_i \in \cU_{0}^{n-1}$ be such that 
\begin{align*}
\ttH (T^i \Phi | (\mathbbm{1}_{\cU})_0^{n-1} ) 
&= \ttH (T^i \Phi | \mathbbm{1}_{U} ) \\
&\leqslant \sum_{\varphi \in \Phi} \sup_{T^{-i} U_i} \varphi \circ T^i = \ttH ( \Phi | \mathbbm{1}_{U_i} ) \leqslant  \ttH ( \Phi | \mathbbm{1}_{\cU} )
\end{align*}
Therefore,
\begin{align*}
\log \ttH (\Phi_0^{n-1}) \leqslant \log \# \cV_n + n \log \ttH ( \Phi | \mathbbm{1}_{\cU}).
\end{align*}
We now estimate the last term. Let $U \in \cU$ be maximizing $\ttH ( \Phi | \mathbbm{1}_{\cU})$ and let $x \in U$. Hence
\begin{align*}
\ttH ( \Phi | \mathbbm{1}_{\cU}) = \sum_{\varphi \in \Phi} \varphi(x) +  \sup_{U} \varphi - \varphi(x) \leqslant 1 +  \left( \sup_{U} \varphi - \inf_{U} \varphi \right) \# \Phi \leqslant 1 + \ve.
\end{align*}
Finally, combining the above and taking the limit in $n$,
\begin{align*}
\tth ( T, \Phi ) \leqslant h(T,\cU) + \log (1 + \ve) \leqslant \htop(T) + \log (1 + \ve).
\end{align*}
As it holds for any $\ve > 0$, we get that $\tth(T,\Phi) \leqslant \htop(T)$ for every continuous partition of unity $\Phi$. Hence $\tthtop(T) \leqslant \htop(T)$.
\end{proof}

\begin{corollary}\label{corol: eq of defs metr and top}
For any topological dynamical system $(X,T)$ and any $T$-invariant measure $\mu$ it holds $h_\mu(T) = \tth_\mu(T)$ and $\htop(T) = \tthtop(T)$.
\end{corollary}

\begin{proof}
From Propositions~\ref{prop: thmu geq hmu} and~\ref{prop: thmu leq hmu}, we get that for any $T$-invariant measure $\mu$, $\tth_{\mu}(T) = h_{\mu}(T)$. Furthermore, using the classical variational principle, as well as Propositions~\ref{prop: thmu leq thtop} and~\ref{prop: thtop leq htop}, we get
\begin{align*}
\htop(T) &= \sup_{\mu} h_{\mu}(T) 
= \sup_{\mu} \tth_{\mu}(T) 
\leqslant \tthtop(T)
\leqslant \htop(T),
\end{align*}
hence $\tthtop(T) = \htop(T)$.
\end{proof}

\subsection{Some other properties}

Here, we prove several properties of the local entropy. We begin to prove that the local metric entropy defined with partitions of unity behaves just like its counterpart with respect to product systems. 

\begin{proposition}\label{prop: entropy of product}
For any continuous self-maps $T$ and $S$ acting respectively on $X$ and $Y$, and any measure $\mu \in \cM( X \times Y, T \times S)$, $\mu_X$ and $\mu_Y$ its marginals on respectively $X$ and $Y$, then for any partitions of unity $\Phi$ and $\Psi$ on respectively $X$ and $Y$, 
\begin{align*}
\max \{ \tth_{\mu_X}(T,\Phi), \, \tth_{\mu_Y}(S,\Psi) \} \leqslant \tth_\mu(T \times S, \Phi \otimes \Psi) \leqslant \tth_{\mu_X}(T,\Phi) + \tth_{\mu_Y}(S,\Psi) .
\end{align*}
\end{proposition}

\begin{proof}
Let $\Phi$ and $\Psi$ be continuous partition of unity on respectively $X$ and $Y$. We extend $\Phi$ into $\bar \Phi = \{ \bar \varphi \mid \varphi \in \Phi \}$ on $X \times Y$ by setting $\bar \varphi(x,y) = \varphi(x)$. Thus $\bar \Phi$ is a partition of unity of $X \times Y$. We proceed analogously to extend $\Psi$ into $\bar \Psi$. Notice that $\Phi \otimes \Psi = \bar \Phi \vee \bar \Psi$, and furthermore for any $n \geqslant 1$,
\begin{align*}
(\Phi \otimes \Psi)_0^{n-1} = (\bar \Phi)_0^{n-1} \vee (\bar \Psi)_0^{n-1}.
\end{align*}
In order to prove the first inequality of the proposition, note that
\begin{align*}
\ttH_\mu((\Phi \otimes \Psi)_0^{n-1}) = \ttH_\mu((\bar \Phi)_0^{n-1}) + \ttH_\mu( (\bar \Psi)_0^{n-1} | (\bar \Phi)_0^{n-1}) \geqslant \ttH_\mu((\bar \Phi)_0^{n-1}) = \ttH_{\mu_X}( \Phi_0^{n-1}),
\end{align*}
where the last equality follows from the definition of $\bar Phi$ and of $\mu_X$. Furthermore, exchanging the roles of $\Phi$ and $\Psi$ in the above leads to
\begin{align*}
\ttH_\mu((\Phi \otimes \Psi)_0^{n-1}) \geqslant \ttH_{\mu_Y}( \Phi_0^{n-1})
\end{align*}
Hence, combining the above and taking the limit as $n$ goes to infinity yield
\begin{align*}
\tth_\mu(T \times S, \Phi \otimes \Psi) \geqslant \max \{ \tth_{\mu_X}(T, \Phi), \, \tth_{\mu_Y}(S, \Psi)  \}.
\end{align*} 

We now turn to the second inequality of the proposition. Note that,
\begin{align*}
\tth_\mu(T\times S, \Phi \otimes \Psi) = \tth_\mu(T\times S, \bar \Phi \vee \bar \Psi) \leqslant \tth_\mu(T\times S, \bar \Phi) + \tth_\mu(T\times S, \bar \Psi).
\end{align*} 
Since for each $\varphi \in \Phi$ and $\bar \varphi \in \bar \Phi$ its extension, we have that $(T \times S)\bar \varphi(x,y) = T \varphi(x)$, we get that for every $n$,
\begin{align*}
\ttH_\mu((\bar \Phi)_0^{n-1}) = \ttH_{\mu_X}(\Phi_0^{n-1}).
\end{align*}
Therefore, $\tth_\mu(T\times S, \bar \Phi) = \tth_{\mu_X}(T,\Phi)$. Exchanging $\Phi$, $\bar \Phi$, $T$ and $\mu_X$ with $\Phi$, $\bar \Psi$, $S$ and $\mu_Y$ in the above yields $\tth_\mu(T\times S, \bar \Psi) = \tth_{\mu_Y}(S,\Psi)$, which concludes the proof.
\end{proof}

We now prove that the local metric entropy function and its conditional version are upper semi-continuous.

\begin{proposition}\label{prop: upper semi-continuity}
For any topological dynamical system $(X,T)$, and any continuous partitions of unity $\Phi$, $\Psi$, the map $\mu \mapsto \tth_\mu(T,\Phi | \Psi)$ is upper semi-continuous on $\cM(X,T)$ with respect to the weak-$^*$ topology.
\end{proposition}  

\begin{proof}
For each $n \geqslant 1$, by continuity of each $\varphi \in \Phi_0^{n-1}$, $\psi \in \Psi_0^{n-1}$ and of $\varphi \psi \log \varphi$, the map $\mu \mapsto \ttH_\mu( \Phi_0^{n-1} | \Psi_0^{n-1})$ is continuous by definition of the weak-$^*$ topology. Now, by sub-additivity, we can write $\mu \mapsto \tth_\mu(T,\Phi,\Psi)$ as the infimum of a countable family of continuous functions defined over a compact set, hence the claimed upper semi-continuity.
\end{proof}

\section{Topological and metric pressures}\label{sect: pressure}

In this section, we extend the definitions of entropies with continuous partitions of unity to pressures, both metric and topological. We prove that these new definition coincides with the usual ones in Theorem~\ref{thm: tPtop = Ptop}.

We begin by defining the metric pressure.

\begin{definition}
Let $g : X \to \mathbbm{R}$ be a continuous potential. Given a continuous partition of unity $\Phi$, the static metric pressure of $(T,g)$ is
\begin{align*}
\ttP_\mu(T,g,\Phi,n) \coloneqq \ttH_\mu(\Phi_0^{n-1}) + \int_X g \intd \mu,
\end{align*}
and its local metric pressure is 
\begin{align*}
\ttP_\mu(T,g,\Phi) \coloneqq \tth_\mu(T,\Phi) + \int_X g \intd \mu.
\end{align*}
The metric pressure of $(T,g)$ is
\begin{align*}
\ttP_\mu(T,g) \coloneqq \sup \{ \ttP_\mu(T,g,\Phi) \mid \text{ $\Phi$ continuous partition of unity} \} = \tth_\mu(T) + \int_X g \intd \mu .
\end{align*}
\end{definition}
Since $\tth_\mu(T) = h_\mu(T)$, the metric pressure $\ttP_\mu(T,g)$ coincides with its classical counterpart denoted $P_\mu(T,g)$.

We now turn to the notion of topological pressure.

\begin{definition}
Let $g : X \to \mathbbm{R}$ be a continuous potential. Given a continuous partition of unity $\Phi$, the static pressure of $(T,g)$ is
\begin{align*}
\ttP(T,g,\Phi,n) \coloneqq \sum_{\varphi \in \Phi_0^{n-1}} \sup \varphi e^{S_n g}, \quad \text{where } S_n g \coloneqq \sum_{i=0}^{n-1} g \circ T^i,
\end{align*}
and its local pressure is 
\begin{align*}
\ttPtop(T,g,\Phi) \coloneqq \lim_{n \to \infty} \frac{1}{n} \log \ttP(T,g,\Phi,n).
\end{align*}
The topological pressure of $(T,g)$ is
\begin{align*}
\ttPtop(T,g) \coloneqq \sup \{ \ttPtop(T,g,\Phi) \mid \text{ $\Phi$ continuous partition of unity} \}.
\end{align*}
\end{definition}

The existence of the limit in the definition of $\ttPtop(T,g)$ comes from the following results.

\begin{lemma}
The sequence $\left( \log \ttP(T,g,\Phi,n) \right)_n$ is sub-additive. In particular, $\ttPtop(T,g,\Phi)$ is well defined.
\end{lemma}

\begin{proof}
For any $n$ and $m$ larger than $1$,
\begin{align*}
\ttP(T,g,\Phi,n+m) &= \sum_{\varphi_1 \circ T^n \varphi_2 \in T^n \Phi_0^{m-1} \vee \Phi_0^{n-1}} \sup \varphi_1 \circ T^n \varphi_2 e^{S_m g \circ T^n + S_n g} \\
&\leqslant \sum_{\varphi_1 \in \Phi_0^{m-1}} \sup \varphi_1 \circ T^n e^{S_m g \circ T^n} \sum_{\varphi_2 \in \Phi_0^{n-1}} \sup \varphi_2 e^{S_n g} \\
&\leqslant \ttP(T,g,\Phi,n) \, \ttP(T,g,\Phi,m).
\end{align*}
Taking logarithm of the above concludes the proof.
\end{proof}

We can now prove that
\begin{theorem}\label{thm: tPtop = Ptop}
For any continuous potential $g$, $\ttPtop(T,g) = \Ptop(T,g)$. Furthermore, for every $T$-invariant measure $\mu$, $\ttP_{\mu}(T,g) = P_{\mu}(T,g)$.
\end{theorem}
We split the proof into several inequalities.

\begin{lemma}
For any continuous potential $g$ and $T$-invariant measure $\mu$, $\ttP_\mu(T,g) \leqslant \ttPtop(T,g)$.
\end{lemma}

\begin{proof}
Applying Lemma~\ref{lemma: classical}, we get
\begin{align*}
\log \ttPtop(T,g,\Phi,n) &= \log \sum_{\varphi \in \Phi_0^{n-1}} e^{\sup( \log \varphi + S_n g)} \geqslant \sum_{\varphi \in \Phi_0^{n-1}} \mu(\varphi) \left( -\log \mu(\varphi) + \sup( \log \varphi + S_n g) \right) \\
&\geqslant \sum_{\varphi \in \Phi_0^{n-1}} - \mu(\varphi) \log \mu(\varphi) + \mu(\varphi \sup( \log \varphi + S_n g) ) \\
&\geqslant \sum_{\varphi \in \Phi_0^{n-1}} - \mu(\varphi) \log \mu(\varphi) + \mu(\varphi (\log \varphi + S_n g) ) = \ttH_{\mu}(\Phi_0^{n-1}) + \mu(S_n g).
\end{align*}
Therefore, by taking the limit in $n$,
\begin{align*}
\ttPtop(T,g,\Phi) \geqslant \tth_{\mu}(T,\Phi) + \mu(g) = \ttP_{\mu}(T,g,\Phi).
\end{align*}
Taking supremum in $\Phi$ yields the claim.
\end{proof}

\begin{lemma}
For any continuous potential $g$, $\ttPtop(T,g) \leqslant \Ptop(T,g)$.
\end{lemma}

\begin{proof}
We proceed very similarly as in the proof of Proposition~\ref{prop: thtop leq htop}. More precisely, let $\Phi$ be a continuous partition of unity, $\ve > 0$ and define $\cU$ the associated open cover of $X$.

For each $n$, let $\cV_n \subset \cU_0^{n-1}$ be a sub-cover minimizing $\sum_{V \in \cV} \sup_{V} e^{S_n g}$. By definition of the usual topological pressure, we get that
\begin{align*}
\lim_{n \to \infty} \frac{1}{n} \log \sum_{V \in \cV_n} \sup_{V} e^{S_n g} = \Ptop(T,g,\cU).
\end{align*} 
Thus,
\begin{align*}
\ttP(T,g,\Phi,n) &= \sum_{\varphi \in \Phi_0^{n-1}} \sup \varphi e^{S_n g} \leqslant \sum_{V \in \cV_n} \sum_{\varphi \in \Phi_0^{n-1}} \sup_V \varphi e^{S_n g} \\
&\leqslant \sum_{V \in \cV_n} \sup_V e^{S_n g} \ttH(\Phi_0^{n-1} | \mathbbm{1}_V ) \leqslant \left( \sum_{V \in \cV_n} \sup_V e^{S_n g} \right) \ttH(\Phi_0^{n-1} | (\mathbbm{1}_{\cU})_0^{n-1} ) \\
& \leqslant \left( \sum_{V \in \cV_n} \sup_V e^{S_n g} \right) (1+\ve)^n .
\end{align*}
Therefore,
\begin{align*}
\ttPtop(T,g,\Phi) \leqslant \Ptop(T,g,\cU) + \log(1+\ve) \leqslant \Ptop(T,g) + \log(1+\ve).
\end{align*}
Hence $\ttPtop(T,g,\Phi) \leqslant \Ptop(T,g)$ for every continuous partition of unity $\Phi$, which yields the claim.
\end{proof}

\begin{proof}[Proof of Theorem~\ref{thm: tPtop = Ptop}]
The variational principle for the pressure as well as the above yield
\begin{align*}
\Ptop(T,g) = \sup_{\mu} (h_\mu(T) + \mu(g)) = \sup_{\mu} ( \tth_\mu(T) + \mu(g)) \leqslant \ttPtop(T,g) \leqslant \Ptop(T,g).
\end{align*}
\end{proof}

\section{Sequences with vanishing diameters}\label{sect: vanish diam}

In this section, we give a sufficient condition for a sequence of continuous partitions of unity to be maximizing with respect to any invariant measure. To do so, we call diameter of a partition of unity $\Phi$ the quantity
\begin{align*}
\diam (\Phi) = \max \{ \diam( \supp( \varphi)) \mid \varphi \in \Phi \},
\end{align*}
where the diameter of a subset of $X$ is computed using a fixed metric on $X$.

\begin{theorem}\label{thm: vanish diam metric}
If $\Psi_n$ is a sequence of continuous partitions of unity such that $\lim_{k \to \infty} \diam (\Psi_k) = 0$, then for each $\mu \in \cM(X,T)$,
\begin{align*}
\lim_{k \to \infty} \tth_{\mu}(T,\Psi_k) = \tth_{\mu}(T).
\end{align*} 
\end{theorem}

The proof relies on the seemingly unrelated following result.

\begin{lemma}
For any $\mu \in \cM(X,T)$, 
\begin{align*}
\tth_\mu(T) = \sup \{ \tth_\mu(T,\Phi) \mid \Phi \text{ positive continuous partition of unity} \},
\end{align*}
where a partition of unity is said to be positive if each of its function is positive.
\end{lemma}

\begin{proof}
One inequality is obvious from the definition of $\tth_\mu(T)$. For the reverse inequality, we proceed as in the proof of Lemma~\ref{prop: thmu geq hmu}. That is, we fix some $\ve_0 >0$ and $\mu \in \cM(X,T)$, we pick a partition $\cP$ with large entropy, we construct another partition $\cQ$ and an open cover $\alpha = \alpha_\delta$ which depends on a parameter $\delta >0$. Finally, for $\Phi=\{\varphi_{-2},\ldots, \varphi_k \}$ a continuous partition of unity subordinated to $\alpha$, we define $\Psi_\delta = \{ \frac{1}{1+(k+3)\delta}(\varphi + \delta) \mid \varphi \in \Phi \}$. Note that $\Psi_\delta$ is a positive and continuous partition of unity. We now estimate $\tth_\mu(T,\Psi_\delta)$. For this, we prove that for $\delta$ small enough,  $\ttH_\mu(\mathbbm{1}_\cQ | \Psi_\delta) < \ve_0$.

Recall that
\begin{align*}
\ttH_\mu(\mathbbm{1}_\cQ | \Psi_\delta) = \sum_{i=-2}^k \mu(\psi_i) \sum_{j=-1}^k - \mu_{\psi_i}(\mathbbm{1}_{Q_j}) \log \mu_{\psi_i}(\mathbbm{1}_{Q_j}).
\end{align*}
We split the sum into the contributions of $\psi_{-2}$, and of $\psi_i$, $i \geqslant -1$.

By uniform continuity of $\eta(x) = -x\log(x)$ on $[0,1]$, there exists $\lambda >0$ such that for all $x,y \in [0,1]$, $|x-y| < \lambda$ implies $|\eta(x)-\eta(y)| < \ve_0 /(4(k+2))$. Up to decreasing the value of $\lambda$, we can assume that $\lambda < \inf_{i \neq -2} (1 - \mu(\mathbbm{1}_{Q_i}))$.

We start with the case $i \neq -2$. For $j \neq i $ we get that
\begin{align*}
\mu_{\psi_i}(\mathbbm{1}_{Q_j}) = \frac{\delta \mu(\mathbbm{1}_{Q_j})}{\mu(\varphi_i) + \delta} < \lambda \Leftarrow \delta < \lambda \inf_{i \neq j} \frac{\mu(\varphi_i)}{\mu(\mathbbm{1}_{Q_j})},
\end{align*}
and for $j=i$,
\begin{align*}
\mu_{\psi_i}(\mathbbm{1}_{Q_i}) = \frac{\mu(\varphi_i) + \delta \mu(\mathbbm{1}_{Q_i}) }{\mu(\varphi_i) + \delta} > 1 - \lambda \Leftarrow \delta < \inf_{i \neq -2}\frac{\mu(\varphi_i)}{1 - \lambda - \mu(\mathbbm{1}_{Q_i})}.
\end{align*}
Therefore, for small enough $\delta$, the contribution of the $\psi_i$, $i \geqslant -1$ is
\begin{align*}
\sum_{i=-1}^k \mu(\psi_i) \sum_{j=-1}^k - \mu_{\psi_i}(\mathbbm{1}_{Q_j}) \log \mu_{\psi_i}(\mathbbm{1}_{Q_j}) 
&< \sum_{i = -1}^{k} \mu(\psi_i) \sum_{j\neq i} \frac{\ve_0}{4(k+2)}  + \sum_{i = -1}^{k} \mu(\psi_i) \frac{\ve_0}{4(k+2)} < \frac{\ve_0}{2}.
\end{align*}
On the other hand, the contribution of $\psi_{-2}$ is
\begin{align*}
\mu(\psi_{-2}) \sum_{j=-1}^k - \mu_{\psi_{-2}}(\mathbbm{1}_{Q_j}) \log \mu_{\psi_{-2}}(\mathbbm{1}_{Q_j}) \leqslant \mu(\psi_{-2}) \log(k+2) \leqslant \frac{\mu(\cN_{\delta}(\partial \cQ))}{1+(k+3)\delta} < \frac{\ve_0}{2},
\end{align*}
where the last inequality holds for $\delta$ small enough. Hence $\ttH_\mu(\mathbbm{1}_\cQ | \Psi_\delta) < \ve_0$. In conclusion, for any $\ve_0$ there exists a positive and continuous partition of unity $\Psi = \Psi_\delta$ such that
\begin{align*}
\tth_\mu(T,\Psi) \geqslant h_\mu(T) - 3 \ve_0,
\end{align*}
hence the claim.
\end{proof}

\begin{proof}[Proof of Theorem~\ref{thm: vanish diam metric}]
Let $\mu \in \cM(X,T)$ and $\Phi$ be a continuous partition of unity. Without loss of generality, we can assume that $\Phi$ is positive. Therefore, for every $k$,
\begin{align*}
\tth_\mu (T,\Phi) - \tth_\mu(T,\Psi_k) 
&= \lim_{n \to \infty} \frac{1}{n} \left( \ttH_\mu(\Phi_0^{n-1}) - \ttH_\mu((\Psi_k)_0^{n-1}) \right) \\
&\leqslant \lim_{n \to \infty} \frac{1}{n} \left( \ttH_\mu((\Phi \vee \Psi_k )_0^{n-1}) - \ttH_\mu((\Psi_k)_0^{n-1}) \right)\\
&\leqslant \lim_{n \to \infty} \frac{1}{n} \ttH_\mu(\Phi_0^{n-1} | (\Psi_k)_0^{n-1}) \leqslant \ttH_\mu(\Phi | \Psi_k ).
\end{align*} 
We now prove that $\ttH_\mu(\Phi | \Psi_k )$ vanishes when $k$ goes to infinity.

Rewriting the definition of the conditional entropy yields
\begin{align*}
\ttH_\mu(\Phi | \Psi_k ) = \mu \left( \sum_{ \varphi \in \Phi} \varphi \log \varphi \right) - \mu \left( \sum_{\varphi \in \Phi} \varphi \sum_{\psi \in \Psi_k} \psi \log \mu_\psi(\varphi) \right). 
\end{align*} 
Therefore, it is enough to prove that for each $\varphi \in \Phi$, the sequence of functions $\sum_{\psi \in \Psi_k} \psi \log \mu_\psi(\varphi)$ converges uniformly to $\log \varphi$.

Let $\ve > 0$ be small enough so that for every $\varphi \in \Phi$, $\inf \varphi > \ve$. Since $X$ is compact, each $\varphi \in \Phi$ is uniformly continuous: let $\delta_\varphi >0$ be associated to $\ve^2$. Denote $\delta = \min_{\varphi \in \Phi} \delta_\varphi > 0$. Take $k$ large enough so that $\diam \Psi_k < \delta$. Therefore, we get that for each $\varphi \in \Phi$ and $\psi \in \Psi_k$,
\begin{align*}
\inf_{\supp \psi} \varphi \leqslant \mu_\psi(\varphi) \leqslant \sup_{\supp \psi} \varphi.
\end{align*}
Hence, for every $x \in \supp \psi$, 
\begin{align*}
\log \frac{\inf_{\supp \psi} \varphi}{\sup_{\supp \psi} \varphi} \leqslant \log \frac{\varphi(x)}{\mu_\psi(\varphi)} \leqslant \log \frac{\sup_{\supp \psi} \varphi}{\inf_{\supp \psi} \varphi}.
\end{align*}
However, since 
\begin{align*}
\sup_{\supp \psi} \varphi -\inf_{\supp \psi} \varphi < \ve^2 < \ve \inf \varphi \leqslant \ve \inf_{\supp \psi} \varphi,
\end{align*}
we get that 
\begin{align*}
\frac{\sup_{\supp \psi} \varphi}{\inf_{\supp \psi} \varphi} < 1 + \ve.
\end{align*}
Thus, 
\begin{align*}
\left| \log \varphi - \sum_{\psi \in \Psi_k} \psi \log \mu_\psi(\varphi) \right | 
= \left| \sum_{\psi \in \Psi_k} \psi \log \frac{\varphi}{\mu_\psi(\varphi)} \right| < \log(1 + \ve) \leqslant \ve.
\end{align*}
In particular, because of the above uniform convergence, for every $k$ large enough we have $\ttH_\mu(\Phi | \Psi_k) \leqslant \ve$. Hence
\begin{align*}
\tth_\mu(T,\Phi) - \ve \leqslant \tth_\mu(T,\Psi_k).
\end{align*}
Taking the appropriate limits yield the claim.
\end{proof}

\section{Topological tail entropy and tail variational principle}\label{sect: top tail and tail var pple}

From the previous section, Theorem~\ref{thm: vanish diam metric} states that the sequence of functions $(\mu \mapsto \tth_\mu(T,\Psi_k))_k$ converges point-wise to $\mu \mapsto \tth_\mu(T)$ whenever $\diam ( \Psi_k )$ converges to zero. In this section, we are interested in the uniform converge, or more precisely its lack of uniformity. The sequence associated to $\Psi_k$ satisfies a tail variational principle in the following sense.
\begin{theorem}\label{thm: tail var ppl}
For any topological dynamical system $(X,T)$, and any sequence of continuous partitions $\Psi_k$, if $\lim_{k \to \infty} \diam ( \Psi_k ) =0$, then
\begin{align*}
\lim_{k \to \infty} \sup_{\mu \in \cM(X,T)} \tth_\mu(T) - \tth_\mu(T,\Psi_k) = h^*(T),
\end{align*}
where $h^*(T)$ is the topological tail entropy, first defined by Misiurewicz \cite{misiurewicz76b}.
\end{theorem}

In order to prove this results, it is more convenient to introduce a definition of the topological tail entropy in terms of continuous partitions of unity. To this end, we introduce a notion of conditional topological entropy with respect to partitions of unity. Several such notions could be used. The one we chose first seems better suited for a local variational principle (discussed in Section~\ref{sect: open}). Another simpler notion is presented in Subsection~\ref{subsect: another top tail entropy}.

\subsection{Conditional and tail topological entropies}

\begin{definition}\label{def: top tail entropy V1}
Let $\Phi$ and $\Psi$ be two continuous partitions of unity. For any $\psi \in \Psi$, define
\begin{align*}
\ttH(\Phi \mid \psi) &\coloneqq \sum_{\varphi \in \Phi} \sup \varphi_{|_{\supp \psi}}.
\end{align*}
For any $n \geqslant 1$, define
\begin{align*}
\ttH_{n}(\Phi \mid \Psi) &\coloneqq \sup_{a \in \cD(\Psi_0^{n-1})} \sum_{\psi \in \Psi_0^{n-1}} a_\psi \ttH(\Phi_0^{n-1} \mid \psi),
\end{align*}
where,
\begin{align*}
\cD(\Psi) &\coloneqq \{ a = (a_\psi)_{\psi \in \Psi} \mid \inf \psi \leqslant a_\psi \leqslant \sup \psi \,\, \text{ and } \sum_{\psi \in \Psi} a_\psi =1 \}.
\end{align*}
The local conditional topological entropy of $(T,\Phi$) given $\Psi$ is
\begin{align*}
\tth(\Phi \mid \Psi) &\coloneqq \limsup_{n \to \infty} \frac{1}{n} \log \ttH_{n}(\Phi \mid \Psi).
\end{align*}
The conditional topological entropy of $T$ given $\Psi$ is
\begin{align*}
\tth(T \mid \Psi) &\coloneqq \sup \{ \tth(\Phi \mid \Psi) \mid \Phi \text{ continuous partition of unity} \}.
\end{align*}
Finally, the topological tail entropy of $T$ is
\begin{align*}
\tth^*(T) &\coloneqq \inf \{ \tth(T \mid \Psi) \mid \Psi \text{ continuous partition of unity} \}.
\end{align*}
\end{definition}

We can relate our new definition of topological tail entropy to the classical one.

\begin{proposition}\label{prop: top tail entropy ineq}
For any continuous self-map $T$ of a compact metric set $X$, $\tth^*(T) \leqslant h^*(T)$.
\end{proposition}

\begin{proof}
Let $\Psi$ be a continuous partition of unity, and let $\cV = \{ \mathring{\supp}(\psi) \mid \psi \in \Psi \}$ be its associated open cover of $X$. Let also $\Phi$ be a continuous partition of unity. Now, notice that for every $n$,
\begin{align}\label{eq: comparison tail entropy}
\sup_{a \in \cD(\Psi_0^{n-1})} \sum_{\psi \in \Psi_0^{n-1}} a_\psi \sum_{\varphi \in \Phi_0^{n-1}} \sup_{\supp \psi} \varphi \leqslant \max_{\psi \in \Psi_0^{n-1}} \sum_{\varphi \in \Phi_0^{n-1}} \sup_{\supp \psi} \varphi \leqslant \max_{V \in \cV_0^{n-1}} \sum_{\varphi \in \Phi_0^{n-1}} \sup_{V} \varphi .
\end{align}
Denote by $V_n$ an element of $\cV_0^{n-1}$ realizing the maximum in the rightmost term. Let $\delta > 0$ and let $\cU$ be an open cover of $X$ with the property that for all $U \in \cU$ and $ \varphi \in \Phi$ 
\begin{align*}
\sup_{U} \varphi - \inf_{U} \varphi < \frac{\delta}{\# \Phi}.
\end{align*}
Let $\cU_n \subset \cU_{0}^{n-1}$ be an open cover of $V_n$ with minimal cardinality. Therefore,
\begin{align*}
\sup_{a \in \cD(\Psi_0^{n-1})} \sum_{\psi \in \Psi_0^{n-1}} a_\psi \sum_{\varphi \in \Phi_0^{n-1}} \sup_{\supp \psi} \varphi 
&\leqslant \sum_{\varphi \in \Phi_0^{n-1}} \sup_{V_n} \varphi 
\leqslant \sum_{U \in \cU_n} \sum_{\varphi \in \Phi_0^{n-1}} \sup_{V_n \cap U} \varphi \leqslant \#\cU_n \, (1+\delta)^n \\
&\leqslant H(\cU_0^{n-1} | V_n) \, (1+\delta)^n \leqslant H(\cU_0^{n-1} | \cV_0^{n-1}) \, (1+\delta)^n .
\end{align*}
By taking the appropriate limits, we get that
\begin{align*}
\tth(\Phi | \Psi ) \leqslant h(\cU | \cV) + \log(1+\delta) \leqslant h(T|\cV) + \delta. 
\end{align*}
Since this holds for any $\Phi$ and $\delta$, it yields
\begin{align*}
\tth(T|\Psi) \leqslant h(T|\cV).
\end{align*}

Now, consider a sequence $\Psi_k$ such that $\lim_k \diam (\Psi_k) = 0$, and let $\cV_k$ be the associated open covers. Then, we also have that $\lim_k \diam (\cV_k) = 0$. By \cite[Theorem~2.1]{misiurewicz76b}, we then know that $h^*(T) = \lim_k h(T|\cV_k)$. Therefore,
\begin{align}\label{eq: encadrement}
\tth^*(T) \leqslant \limsup_{k \to \infty} \tth(T|\Psi_k) \leqslant h^*(T),
\end{align}
which ends the proof.
\end{proof}

\begin{proposition}\label{prop: tail entropy bounds diff}
For any continuous partition of unity $\Psi$, it holds
\begin{align*}
\sup_{\mu \in \cM(X,T)} \tth_\mu(T) - \tth_\mu(T,\Psi) \leqslant \tth(T|\Psi).
\end{align*}
\end{proposition}

\begin{proof}
Let $\mu \in \cM(X,T)$, $\ve_0 > 0$ and chose some partition of unity $\Phi$ such that 
\begin{align*}
\tth_\mu(T,\Phi) \geqslant \tth_\mu(T) - \ve_0.
\end{align*}
Now, since for every $n\geqslant 1$, we have that
\begin{align*}
\log \sum_{\psi \in \Psi_0^{n-1}} \mu(\psi) &\sum_{\varphi \in \Phi_0^{n-1}} \sup_{\supp \psi} \varphi 
= \log \sum_{\psi \in \Psi_0^{n-1}} \sum_{\varphi \in \Phi_0^{n-1}} \exp \left( \log \mu(\psi) + \log \sup_{\supp \psi} \varphi \right) \\
&\geqslant \sum_{\psi \in \Psi_0^{n-1}} \sum_{\varphi \in \Phi_0^{n-1}} \mu(\varphi \psi) \left( - \log \mu(\varphi \psi) + \log \mu(\psi) + \log \sup_{\supp \psi} \varphi \right) \\
&\geqslant \sum_{\psi \in \Psi_0^{n-1}} \sum_{\varphi \in \Phi_0^{n-1}} -\mu(\varphi \psi) \log \frac{\mu(\varphi \psi)}{\mu(\psi)} + \mu(\psi \varphi \log \varphi) = \ttH_\mu(\Phi_0^{n-1} | \Psi_0^{n-1} ),
\end{align*}
and for $a_\psi= \mu(\psi)$, $\psi \in \Psi_0^{n-1}$ we have that $a \in \cD(\Psi_0^{n-1})$, it follows by taking the appropriate limits
\begin{align}\label{eq: partial local tail var ppl}
\tth_\mu(\Phi | \Psi) \leqslant \tth(\Phi | \Psi).
\end{align}
Finally,
\begin{align*}
\tth_\mu(T) - \tth_\mu(T,\Psi) &\leqslant \tth_\mu(T,\Phi) - \tth_\mu(T,\Psi) + \ve_0 \leqslant \tth_\mu(\Phi | \Psi) + \ve_0 \leqslant \tth(\Phi | \Psi) + \ve_0 \\
&\leqslant \tth(T| \Psi) + \ve_0,
\end{align*}
for any $\ve_0 > 0$ and any $\mu \in \cM(X,T)$, which ends the proof.
\end{proof}

We can now prove the tail variational principle.

\begin{proof}[Proof of Theorem~\ref{thm: tail var ppl}]
According to Downarowicz and Huczek's \cite[Theorem~3.1]{DH13}, there exists a zero-dimensional principal extension $(X',T')$ of $(X,T)$, that is there is a continuous surjection $\pi : X' \to X$ such that $T \circ \pi = \pi \circ T'$ where $T'$ is a continuous self-map on the compact metric zero-dimensional set $X'$. By zero-dimensional is meant that $X'$ admits a base consisting of sets which are both closed and open (called \emph{clopen}). By principal is meant that for any $\mu \in \cM(X',T')$, $h_\mu(T') = h_{\pi_* \mu}(T)$.

Let $\Phi_k$ be a sequence of continuous partitions of unity on $X$. Therefore $\pi \Phi_k$ is a sequence of \emph{continuous} partitions of unity on $X'$. By definition of the image measure, we get that
\begin{align*}
\tth_\mu(T', \pi \Phi_k) = \tth_{\pi_* \mu}(T,\Phi_k).
\end{align*}

Furthermore, since $X'$ is compact and zero-dimensional, there exists a sequence of partitions $\cP_k$ of $X'$ into clopen sets such that $\cP_{k+1}$ is finer than $\cP_k$ and $\diam(\cP_k)$ converges to zero. Denote $\mathbbm{1}_{\cP_k}$ the associated partitions of unity. Notice that since the partitions are clopen, the partitions of unity $\mathbbm{1}_{\cP_k}$ are continuous.

As in the proof of Proposition~\ref{prop: thmu leq hmu}, for any $k$, there exists $n_k$ such that uniformly in $\mu \in \cM(X',T')$,
\begin{align*}
\tth_\mu(T',\pi \Phi_k) \leqslant \tth_\mu(T',\mathbbm{1}_{\cP_{n_k}}) + 2^{-k}.
\end{align*}
Without loss of generality, we can assume that $(n_k)_k$ is strictly increasing.

Using both Burguet's \cite[Main Theorem]{burguet09} and \cite[Lemma~5(4)]{burguet09}, it follows that $h^*(T) = h^*(T')$ and 
\begin{align*}
h^*(T') = \lim_{k \to \infty} \sup_{\mu \in \cM(X',T')} h_\mu(T') - h_\mu(T',\cP_k).
\end{align*}
Since $h_\mu(T',\cP_k) = \tth_\mu(T',\mathbbm{1}_{\cP_k})$, by combining the above we have that
\begin{align}\label{eq: reverse ineq}
\begin{split}
h^*(T) &= h^*(T') = \lim_{k \to \infty} \sup_{\mu \in \cM(X',T')} h_\mu(T') - \tth_\mu(T',\mathbbm{1}_{\cP_{n_k}}) \\
&\leqslant \liminf_{k \to \infty} \sup_{\mu \in \cM(X',T')} h_\mu(T') - \tth_\mu(T',\pi \Phi_k) + 2^{-k} \\
&\leqslant \liminf_{k \to \infty} \sup_{\mu \in \cM(X',T')} \tth_{\pi_* \mu}(T) - \tth_{\pi_* \mu}(T,\Phi_k) + 2^{-k} \\
&\leqslant \liminf_{k \to \infty} \sup_{\mu \in \cM(X,T)} \tth_{\mu}(T) - \tth_{\mu}(T,\Phi_k) + 2^{-k} \\
&\leqslant \limsup_{k \to \infty} \sup_{\mu \in \cM(X,T)} \tth_{\mu}(T) - \tth_{\mu}(T,\Phi_k) + 2^{-k} 
\leqslant \limsup_{k \to \infty} \tth(T|\Phi_k) + 2^{-k}.
\end{split}
\end{align}
Now, assuming that the $\diam(\Phi_k)$ converges to $0$ we can use \eqref{eq: encadrement}, which finishes the proof.
\end{proof}

From the end of the above proof, we can deduce
\begin{corollary}\label{corollary: equiv top tail entropy}
For any continuous self map $T$ on a compact metric set $X$, it holds 
\begin{align*}
\tth^*(T) = h^*(T).
\end{align*}
\end{corollary}

\begin{proof}
In the proof of Theorem~\ref{thm: tail var ppl}, the assumption on the diameter of $\Phi_k$ converging to zero is only used at the very end. Instead, by taking $\Phi_k$ so that $\tth(T | \Phi_k)$ converges to $\tth^*(T)$, it follows from \eqref{eq: reverse ineq} that $h^*(T) \leqslant \tth^*(T)$. Combined this inequality with Proposition~\ref{prop: top tail entropy ineq} finishes the proof.
\end{proof}

\subsection{The case of vanishing diameters}

Similarly to Section~\ref{sect: vanish diam}, we are interested in approaching $\tth(T,\Psi)$ and $\tth^*(T)$ using sequences of continuous partitions of unity with diameters converging to zero.

\begin{proposition}\label{prop: vanish approach inf}
For any sequence $\Phi_k$ of continuous partitions of unity with $\diam(\Phi_k)$ converging to $0$, it holds that
\begin{align*}
\lim_{k \to \infty} \tth(T| \Phi_k) = \tth^*(T).
\end{align*}
\end{proposition}
\begin{proof}
It follows directly from Corollary~\ref{corollary: equiv top tail entropy} and \eqref{eq: encadrement}.
\end{proof}

\begin{proposition}\label{prop: vanish approach sup}
For any sequence $\Phi_k$ of continuous partitions of unity with $\diam(\Phi_k)$ converging to $0$ and any continuous partition of unity $\Psi$, it holds that 
\begin{align*}
\lim_{k \to \infty} \tth(\Phi_k | \Psi) = \tth(T | \Psi).
\end{align*}
\end{proposition}

Similarly to Theorem~\ref{thm: vanish diam metric}, we will make use of the following seemingly unrelated lemma.
\begin{lemma} For any continuous partition of unity $\Psi$, it holds
\begin{align*}
\tth(T|\Psi) = \sup \{ \tth(\Phi | \Psi) \mid \Phi \text{ positive and continuous partition of unity} \}.
\end{align*} 
\end{lemma}

\begin{proof}
It is obvious from the definition that $\tth(T|\Psi) \geqslant \sup \{ \tth(\Phi | \Psi) \mid \Phi > 0 \}$. We now prove the reverse inequality.

Let $\ve_0 > 0$ and chose some continuous partition of unity $\Phi = \{ \varphi_1, \ldots , \varphi_N \}$ such that 
\begin{align*}
\tth(\Phi | \Psi ) \geqslant \tth(T|\Psi) - \ve_0.
\end{align*}
For any $\delta > 0$, let $\Phi_\delta = \{ \varphi_{i,\delta} = \frac{\varphi_i + \delta/N}{1+\delta} \mid 1 \leqslant i \leqslant N \}$. Notice that $\Phi_\delta$ is a positive and continuous partition of unity. Denote $\sigma : \Phi \to \Phi_\delta$ the bijective map such that $\sigma(\varphi_i) = \varphi_{i,\delta}$. This map can be extended into the bijective map $\sigma_n : \Phi_0^{n-1} \to (\Phi_\delta)_0^{n-1}$ in a natural way.

Therefore, for every $\varphi = \prod_{j=0}^{n-1} \varphi_{i_j} \circ T^j \in \Phi_0^{n-1}$, we have that
\begin{align*}
\frac{\varphi}{\sigma_n(\varphi)} = \prod_{j=0} (1+\delta)\frac{\varphi_{i_j}}{\varphi_{i_j} + \delta} \leqslant (1+\delta)^n,
\end{align*}
that is $\varphi \leqslant (1+\delta)^n \sigma_n(\varphi)$. From this it follows by taking the appropriate limits that
\begin{align*}
\tth(\Phi | \Psi) \leqslant \tth(\Phi_\delta | \Psi) + \log(1+\delta) \leqslant \tth(\Phi_\delta | \Psi) + \delta
\end{align*}
Since $\ve$ and $\delta$ are arbitrary, the claim follows.
\end{proof}

In order to relate the quantities $\tth(\Phi | \Psi)$, where $\Phi > 0$ or $\diam(\Phi)$ small enough, we rely on an approximation result. To state it, we need the following operation on partitions of unity. 
\begin{definition}
For $\Phi$ a partition of unity and $k \geqslant 1$, we define $k \cdot \Psi$ as the partition of unity made of $k$ copies of $\{ \psi /k \mid \psi \in \Psi \}$.
\end{definition}

\begin{lemma}\label{lemma: approx small diam}
Given a continuous partition of unity $\Phi = \{ \varphi_i \mid 1 \leqslant i \leqslant N \}$, for every $k \geqslant 1$ there exists $\delta > 0$ such that for every partition of unity $\Psi$ with $\diam ( \Psi ) < \delta$, there exists a partition $k \cdot \Psi = \bigsqcup_{1 \leqslant i \leqslant N} A_i$, such that for every $i$ 
\begin{align*}
\left| \varphi_i - \sum_{\psi \in A_i} \psi \right| \leqslant \frac{2}{k}.
\end{align*}
\end{lemma}

\begin{proof}
Let $\Phi$ and $\Psi$ two continuous partitions of unity, with $\diam \Psi < \delta$, where $\delta > 0$ is to be determined later.

We start by approximating one function of $\Phi$. Fix some $\varphi \in \Phi$. Denote $\varphi_{[k]} = \sum_{l=0}^{k-1} \frac{1}{k} \mathbbm{1}_{(\varphi > l/k)}$. It follows that
\begin{align*}
\varphi \leqslant \varphi_{[k]} \leqslant \varphi + \frac{1}{k}
\end{align*}
For each $0 \leqslant l < k$, define $\Lambda_{l,k} = \{ \psi \in \Psi \mid \supp(\psi) \cap (\varphi > l/k) \neq \emptyset \}$. Since $\diam \Psi < \delta$, it follows that
\begin{align*}
\mathbbm{1}_{(\varphi > l/k)} \leqslant \sum_{\psi \in \Lambda_{l,k}} \psi \leqslant \mathbbm{1}_{\cN_\delta (\varphi > l/k)}.
\end{align*}
Define $\varphi_{[k,\delta]} = \sum_{l=0}^{n-1} \frac{1}{k} \mathbbm{1}_{\cN_\delta (\varphi > l/k)}$. Therefore $\varphi_{[k,\delta]} \geqslant \varphi_{[k]} \geqslant \varphi$. Furthermore, taking $\delta$ associated by the uniform continuity of $\varphi$ and $1/3k$, we get that for each $x \in \cN_\delta(\varphi > l/k) \smallsetminus (\varphi > l/k)$, $\varphi(x) \in \left[ \frac{l - 1/3}{k}, \frac{l + 1/3}{k} \right]$. In particular, the sets $\cN_\delta(\varphi > l/k) \smallsetminus (\varphi > l/k)$ are disjoint. Thus
\begin{align*}
\varphi_{[k,\delta]} - \varphi 
&\leqslant \varphi_{[k,\delta]} - \varphi_{[k]} + \frac{1}{k}
\leqslant \frac{1}{k} +  \sum_{l=0}^{k-1} \frac{1}{k} \mathbbm{1}_{\cN_\delta(\varphi > l/k) \smallsetminus (\varphi > l/k)} \leqslant \frac{2}{k}
\end{align*} 

Thus, letting $\psi_\varphi = \sum_{l=0}^{k-1} \frac{1}{k} \sum_{\psi \in \Lambda_{l,k}} \psi$, we get that $\psi_\varphi$ is obtained by summing elements of $k\cdot \Psi$ and that $0 \leqslant \psi_\varphi - \varphi \leqslant 2/k$.

 We now approximate the whole $\Phi = \{ \varphi_1, \ldots \varphi_N \}$. For this denote $\Phi' = \{ \varphi'_1, \ldots \varphi'_N \}$, where $\varphi'_1 = \varphi_1$ and $\varphi'_{i} = \varphi'_{i-1} + \varphi_i$ for $i>1$. In particular $\varphi'_N \equiv 1$. Each element of $\Phi'$ is uniformly continuous on $X$. Take $\delta$ associated to all those function and $1 / 3k$. For each $i$, there is by the above construction a function $\psi'_i$ obtained by summing together the elements of $A_i \subset k \cdot \Psi$ and such that $0 \leqslant \psi'_i - \varphi'_i \leqslant 2/k$. Since $1 = \varphi'_N \leqslant \psi'_N \leqslant 1$, we get that $A_n = k \cdot \Psi$. 
 
Furthermore, since the family of functions $\Phi'$ is increasing, we can assume without loss of generality that $A_i \subset A_{i+1}$ for every $i$. Now, define $\psi_1 = \psi'_1$ and $\psi_i = \psi'_i - \psi'_{i-1}$ for every $i > 1$, as well as $\Psi_{k,\delta} = \{ \psi_1, \ldots , \psi_N \}$. Therefore, we get that for every $i$, $|\psi_i - \varphi_i| \leqslant 2/k$, and that $\Psi_{k,\delta}$ is obtained by partitioning $k \cdot \Psi$ into subfamilies and summing their elements together. Hence the claim.
\end{proof}

\begin{proof}[Proof of Proposition~\ref{prop: vanish approach sup}]
First, assume that $\tth(T|\Psi) > 0$ (otherwise the result is straightforward). Let $\ve_0 > 0$ small enough and chose $\Phi = \{ \varphi_i \mid 1 \leqslant i \leqslant N \}$ such that $\tth(\Phi | \Psi) \geqslant \tth(T | \Psi) - \ve_0 > 0$. Without loss of generality, we can assume that $\Phi$ is positive. 

Fix $k$ large enough so that $2/k < \ve_0 \min \{ \inf \varphi \mid \varphi \in \Phi \}$ and let $\delta >0$ from Lemma~\ref{lemma: approx small diam}. Fix some large enough $m$ so that $\diam (\Phi_m) < \delta$. Therefore, there exists $\Phi_{k,m}$ approximating $\Phi$ and obtained from $k \cdot \Phi_m$. Denote by $\sigma : \Phi \to \Phi_{k,m}$ the bijection sending a function to its approximation. We want to compare $\tth(\Phi | \Psi)$ to $\tth(\Phi_{k,m} | \Psi)$.

We can naturally extend $\sigma$ into $\sigma_n : \Phi_0^{n-1} \to (\Phi_{k,m})_0^{n-1}$. Therefore, for every $\varphi = \prod_{j=0}^{n-1} \varphi_{i_j} \circ T^j \in \Phi_0^{n-1}$, we have that
\begin{align*}
\frac{\sigma_n(\varphi)}{\varphi} = \prod_{j=0}^{n-1} \left( 1 + \frac{\sigma_n(\varphi_{i_j}) - \varphi_{i_j}}{\varphi_{i_j}} \right) \circ T^j \geqslant (1-\ve_0)^n,
\end{align*}
that is, $\varphi \leqslant (1-\ve_0)^{-n} \sigma_n(\varphi)$. From this, it follows by taking the appropriate limits that
\begin{align*}
\tth(\Phi | \Psi) \leqslant \tth(\Phi_{k,m} | \Psi) - \log(1-\ve_0) \leqslant \tth(\Phi_{k,m} | \Psi) + 2 \ve_0,
\end{align*}
where the last inequality holds for $\ve_0 < 1/2$.

We now compare $\tth(\Phi_{k,m}|\Psi)$ to $\tth(\Phi_{m}|\Psi)$. First, notice that $(k \cdot \Phi_m)_{0}^{n-1} = k^n \cdot (\Phi_m)_{0}^{n-1}$. Therefore, by grouping together the identical terms in the summation, we get that for every $\psi \in \Psi_0^{n-1}$,
\begin{align*}
\sum_{\varphi \in (k \cdot \Phi_m)_0^{n-1}} \sup_{\supp \psi} \varphi = \sum_{\varphi \in (\Phi_m)_0^{n-1}} \sup_{\supp \psi} \varphi.
\end{align*}
It easily follows by taking the appropriate limits that
\begin{align*}
\tth(k\cdot \Phi_m | \Psi) = \tth( \Phi_m | \Psi).
\end{align*}
Furthermore, since $(\Phi_{k,m})_0^{n-1}$ is obtained from partitioning $(k \cdot \Phi_m)_0^{n-1}$, it follows that for any $\psi \in \Psi_0^{n-1}$,
\begin{align*}
\sum_{\varphi \in (\Phi_{k,m})_0^{n-1}} \sup_{\supp \psi} \varphi \leqslant \sum_{\varphi \in (k \cdot \Phi_m)_0^{n-1}} \sup_{\supp \psi} \varphi.
\end{align*}
Once again, taking the appropriate limits yields
\begin{align*}
\tth(\Phi_{k,m} | \Psi) \leqslant \tth(k \cdot \Phi_m | \Psi).
\end{align*}
Combining the above, we finally get that for any $\ve_0$ small enough there exists $m_0$ such that for every $m \geqslant m_0$,
\begin{align*}
\tth( \Phi_m | \Psi) \geqslant \tth(\Phi_{k,m} | \Psi) \geqslant \tth(\Phi | \Psi) -2 \ve_0 \geqslant  \tth(T | \Psi) - 3\ve_0,
\end{align*}
which ends the proof.
\end{proof}

As a direct consequence of Propositions~\ref{prop: vanish approach inf} and \ref{prop: vanish approach sup}, we get the following corollary.

\begin{corollary}\label{corollary: top tail entropy as double limit V1}
For any sequence $\Phi_n$ of continuous partitions of unity such that $\diam(\Phi_n)$ converges to $0$, it holds
\begin{align*}
h^*(T) = \lim_{n \to \infty} \lim_{k \to \infty} \tth(T, \Phi_k | \Phi_n ).
\end{align*}
\end{corollary}

\subsection{An alternative definition}\label{subsect: another top tail entropy}

We can recover $h^*(T)$ from a seemingly simpler conditional topological entropy using Definition~\ref{def: local cond top entropy V1}.
\begin{definition}
For any continuous partitions of unity $\Phi$ and $\Psi$, define
\begin{align*}
\tth_+(T,\Phi | \Psi) &\coloneqq \lim_{n \to \infty} \frac{1}{n} \log \ttH(\Phi_0^{n-1} | \Psi_0^{n-1}), \\
\tth_+^*(T) &\coloneqq \inf_{\Psi} \sup_{\Phi} \tth_+(T,\Phi | \Psi),
\end{align*} 
where the $\inf$ and $\sup$ are taken over continuous partitions of unity.
\end{definition}
Notice that in the definition of $\tth_+(T,\Phi | \Psi)$, the dependence on $\Psi$ is only through the supports of the $\psi \in\Psi_0^{n-1}$, not on their values. We first check that the limit in the definition exists.

\begin{proposition}
The sequence $n \mapsto \log \ttH(\Phi_0^{n-1} | \Psi_0^{n-1})$ is sub-additive. In particular, the sequence $\frac{1}{n} \log \ttH(\Phi_0^{n-1} | \Psi_0^{n-1})$ converges.
\end{proposition}

\begin{proof}
Writing any $\varphi \in \Phi_0^{n+m-1}$ in the form $\varphi = \varphi_2 \circ T^n \, \varphi_1$ with $\varphi_1 \in \Phi_0^{n-1}$ and $\varphi_2 \in \Phi_0^{m-1}$, and similarly for $\psi \in \Psi_0^{n+m-1}$, it follows that
\begin{align*}
\sup \varphi |_{\supp \psi} &\leqslant \sup \varphi_2 \circ T^n |_{\supp \psi} \sup \varphi_1 |_{\supp \psi} \\
&\leqslant \sup \varphi_2 \circ T^n |_{\supp \psi_2 \circ T^n} \sup \varphi_1 |_{\supp \psi} = \sup \varphi_2 |_{\supp \psi_2} \sup \varphi_1 |_{\supp \psi_1}.
\end{align*}
Therefore
\begin{align*}
\sum_{\varphi \in \Phi_0^{n+m-1}} \sup \varphi |_{\supp \psi} &\leqslant \sum_{\varphi_1 \in \Phi_0^{n-1}} \sup \varphi_1 |_{\supp \psi_1} \sum_{\varphi_2 \in \Phi_0^{m-1}} \sup \varphi_2 |_{\supp \psi_2} \\
&\leqslant \ttH( \Phi_0^{n-1} | \psi_1) \ttH( \Phi_0^{m-1} | \psi_2) \leqslant \ttH( \Phi_0^{n-1} | \Psi_0^{n-1}) \ttH( \Phi_0^{m-1} | \Psi_0^{m-1}).
\end{align*}
Since this holds for every $\psi \in \Psi_0^{n+m-1}$, the sub-additivity follows.
\end{proof}

\begin{proposition}
$\tth_+^*(T) = h^*(T)$.
\end{proposition}

\begin{proof}
It is enough to prove that $\tth^*(T) \leqslant \tth_+^*(T) \leqslant h^*(T)$. This is a direct consequence of \eqref{eq: comparison tail entropy}.
\end{proof}

\begin{corollary}
For any sequence $\Phi_n$ of continuous partitions of unity such that $\diam(\Phi_n)$ convergences to $0$, it holds
\begin{align*}
h^*(T) = \lim_{n \to \infty} \lim_{k \to \infty} \tth_+(T,\Phi_k | \Phi_n).
\end{align*}
\end{corollary}

\begin{proof}
This is a direct consequence of \eqref{eq: encadrement} and Corollary~\ref{corollary: top tail entropy as double limit V1}.
\end{proof}

\section{Topological tail pressure and variational principle}\label{sect: top tail pressure}

In this section, we explain how to extend the results from Section~\ref{sect: top tail and tail var pple} to the case of the topological tail pressure. We will rely on the work of Li, Chen and Cheng~\cite{LCC12}, where they give several equivalent definitions of the topological tail pressure $P^*(T,g)$ of a potential $g$ in terms of separated sets, spanning sets and sequences of partitions. They also proved a tail variational principle for pressure.

We give a new definition of the topological tail pressure of a continuous potential in terms of continuous partitions of unity, and we extend the tail variational principle to sequences $\tth_\mu(T,\Psi_k)$ with $\diam(\Psi_k)$ converging to zero. More precisely,

\begin{definition}[Topological tail pressure]
Let $\Phi$ and $\Psi$ be two continuous partitions of unity. 
For any $n \geqslant 1$, define
\begin{align*}
\ttP_{n}(T,g,\Phi \mid \Psi) &\coloneqq \sup_{a \in \cD(\Psi_0^{n-1})} \sum_{\psi \in \Psi_0^{n-1}} a_\psi \sum_{\varphi \in \Phi_0^{n-1}} \sup_{\supp \psi} \varphi \, e^{S_n g},
\end{align*}
where $\cD(\Psi)$ is defined as in Definition~\ref{def: top tail entropy V1}.\\
The local topological pressure of $(T,g,\Phi$) given $\Psi$ is
\begin{align*}
\ttP(T,g,\Phi \mid \Psi) &\coloneqq \limsup_{n \to \infty} \frac{1}{n} \log \ttP_{n}(T,g,\Phi \mid \Psi).
\end{align*}
The topological pressure of $(T,g)$ given $\Psi$ is
\begin{align*}
\ttP(T,g \mid \Psi) &\coloneqq \sup \{ \ttP(T,g,\Phi \mid \Psi) \mid \Phi \text{ continuous partition of unity} \}.
\end{align*}
Finally, the topological tail pressure of $(T,g)$ is
\begin{align*}
\ttP^*(T,g) &\coloneqq \inf \{ \ttP(T,g \mid \Psi) \mid \Psi \text{ continuous partition of unity} \}.
\end{align*}
\end{definition}

\begin{theorem}\label{thm: tail var pple for pressure}
For any topological dynamical system $(X,T)$ and continuous potential $g : X \to \bR$, it holds that $\ttP^*(T,g) = P^*(T,g)$. \\
Furthermore, for any sequence $\Psi_k$ of continuous partitions of unity with $\diam(\Psi_k)$ converging to zero, it holds
\begin{align*}
\lim_{k \to \infty} \sup_{\mu \in \cM(X,T)} h_\mu(T) - \tth_\mu(T,\Psi_k) + \int_X g \intd \mu = P^*(T,g).
\end{align*}
\end{theorem}

The proof follows the same steps as the ones of Theorem~\ref{thm: tail var ppl} and Corollary~\ref{corollary: equiv top tail entropy}. We summarise the intermediate results in the following

\begin{proposition}
Let $(X,T)$ be a topological dynamical system, and $g$ a continuous potential on $X$.

\noindent
\textit{i)} If $\pi : (X',T') \to (X,T)$ is a zero-dimensional principal extension, then \[ P^*(T',g\circ \pi) = P^*(T,g). \]

\noindent
\textit{ii)} For any continuous partition of unity $\Psi$,
\begin{align*}
\sup_{\mu \in \cM(X,T)} h_\mu(T) - \tth_\mu(T,\Psi) + \int g \intd \mu \leqslant \ttP(T,g \mid \Psi).
\end{align*}

\noindent
\textit{iii)} It holds that $\ttP^*(T,g) \leqslant P^*(T,g)$.
\end{proposition}

\begin{proof}
Point \textit{i)} follows directly from the characterization of $P^*(T,g)$ given by \cite[Lemma~4.4]{LCC12}. For Point~\textit{ii)}, we proceed exactly as in the proof of Proposition~\ref{prop: tail entropy bounds diff}. For Point~\textit{iii)}, we use (as an intermediate step) a definition of $P^*(T,g)$ involving open covers, which is not provided in \cite{LCC12}, and relate it to a definition involving separated sets. We start similarly as in the proof of Proposition~\ref{prop: top tail entropy ineq}. 

Let $\phi$ and $\Psi$ be continuous partitions of unity. Let $\cV = \{ \mathring{\supp}(\psi) \mid \psi \in \Psi \}$ be an open cover of $X$. Therefore,
\begin{align*}
\ttP_n(T,g,\Phi \mid \Psi) \leqslant \max_{\psi \in \Psi_0^{n-1}} \sum_{\varphi \in \Phi_0^{n-1}} \sup_{\supp \psi} \varphi \, e^{S_n g} \leqslant \max_{V \in \cV_0^{n-1}} \sum_{\varphi \in \Phi_0^{n-1}} \sup_{V} \varphi \, e^{S_n g} 
\end{align*}
Denote by $V_n$ the element of $\cV_0^{n-1}$ maximizing the rightmost term. Let $\gamma > 0$ and let $\cU$ be an open cover of $X$ with the property: for all $\varphi \in \Phi$ and $U \in \cU$,
\begin{align*}
\sup_{U} \varphi - \inf_{U} \varphi < \frac{\gamma}{\# \Phi}.
\end{align*}
Furthermore, let $\cU_n \subset \cU_0^{n-1}$ be an open cover of $V_n$ minimizing the quantity $\sum_{U \in \cU_n} \sup_{U} e^{S_n g}$. Therefore,
\begin{align*}
\ttP_n(T,g,\Phi \mid \Psi) &\leqslant \sum_{U \in \cU_n} \sum_{\varphi \in \Phi_0^{n-1}} \sup_{ V_n \cap U} \varphi \, e^{S_n g} \leqslant \sum_{U \in \cU_n} \sup_{U} e^{S_n g} \!\!\!\! \sum_{\varphi \in \Phi_0^{n-1}} \sup_{U} \varphi \leqslant (1+\gamma)^n \sum_{U \in \cU_n} \sup_{U} e^{S_n g}.
\end{align*}
Now, let $\ve \geqslant 2 \diam(\cV)$ and let $\delta$ be a Lebesgue number for $\cU$. Denote $\tau_\cU$ the quantity
\begin{align*}
\tau_{\cU} = \sup \{ |g(x) - g(y)| \mid d(x,y) \leqslant \diam(\cU) \}.
\end{align*}
Therefore, for any $x \in V_n$, we have that $V_n \subset B(x,\ve,n)$ (the $n$-th Bowen ball centred at $x$ with radius $\ve$). Furthermore, if $F$ is an $(n,\delta/2)$-spanning set of $V_n$, then for each $x \in F$, $B(x,\delta/2,n)$ is a subset of a member of $\cU_n$, it follows that 
\begin{align*}
\ttP_n(T,g,\Phi \mid \Psi) &\leqslant (1+ \gamma)^n e^{n \tau_{\cU}} \inf \{ \sum_{x \in F} e^{S_n g(x)} \mid F \, (n, \frac{\delta}{2})\text{-spanning set of }V_n \} \\
&\leqslant (1+ \gamma)^n e^{n \tau_{\cU}} P_n(T,g,\delta/2,\ve),
\end{align*}
where as defined in \cite{LCC12}
\begin{align*}
P_n(T,g,\delta,\ve) \coloneqq  \sup_{x \in X } \sup \left\lbrace \sum_{x \in E} e^{S_n g(y)} \mid E \, \left( n, \delta \right) \text{-separated subset of }B(x,\ve,n) \right\rbrace,
\end{align*}
where the last inequality follows from classical arguments. Notice that we can decrease the diameter of $\cU$ (and thus the value of $\delta$) in the above, yielding a bound uniform in $\Phi$. That is, taking first the $\log$ and the $\limsup$ in $n$,
\begin{align*}
\ttP^*(T,g) \leqslant \ttP(T,g \mid \Psi) = \sup_{\Phi} \ttP(T,g,\Phi \mid \Psi) \leqslant \log(1 + \gamma) + \lim_{\delta \to 0} \limsup_{n \to \infty} \frac{1}{n} \log P_n(T,g,\delta/2,\ve).
\end{align*}
The term in $\tau_\cU$ vanishes because of the continuity of $g$. Now, replacing $\Psi$ by a sequence $\Psi_k$ with $\diam(\Psi_k) = \ve_k/2$ converging to zero, we get that 
\begin{align}\label{eq: encadrement tail pressure}
\begin{split}
\ttP^*(T,g) \leqslant \liminf_{k \to \infty} \ttP(T,g \mid \Psi_k) &\leqslant \log(1+ \gamma) + \lim_{\ve \to 0} \lim_{\delta \to 0} \limsup_{n \to \infty} \frac{1}{n} \log P_n(T,g,\delta,\ve) \\
&\leqslant \log(1+\gamma) + P^*(T,g),
\end{split}
\end{align}
where the last inequality follows from \cite[Lemma~3.3]{LCC12}. Since this holds for any $\gamma > 0$, Point~\textit{iii)} follows as claimed.
\end{proof}

We can now prove the theorem of this section.
\begin{proof}[Proof of Theorem~\ref{thm: tail var pple for pressure}]
The proof fo the variational principle is exactly the same as the one of Theorem~\ref{thm: tail var ppl}, replacing the use in the end of \eqref{eq: encadrement} by \eqref{eq: encadrement tail pressure}. The proof of the equivalence of the definitions is exactly as the one of Corollary~\ref{corollary: equiv top tail entropy}, using \eqref{eq: encadrement tail pressure} instead of \eqref{eq: encadrement}.
\end{proof}

\section{Entropy structures and weak entropy structures}\label{sect: entropy structure}

In this section, we use the notion of \emph{entropy structures} developed by Downarowicz~\cite{downarowicz05}. It is a ``master invariant" in the sense that it determines almost all previously known entropy invariants such as topological entropy, metric entropy function, tail entropy, the symbolic extension entropy and the symbolic extension entropy map -- see for example Downarowicz's book for a complete exposure. More precisely, we prove that if $\Phi_k$ is an increasing sequence of continuous partition of unity with $\diam(\Phi_k)$ converging to $0$, then it determines a sequence of maps uniformly equivalent to the entropy structure (Proposition~\ref{prop: increasing entropy structure with pu}). In particular, we recover the tail variational principle in this case (Corollary~\ref{corollary: tail var pple}).

Furthermore, by modifying slightly the equivalence relation, that result extends to sequences $\Phi_k$ not necessarily increasing (Theorem~\ref{thm: weak entropy structure}). Elements in the now larger equivalence class shares the same superenvelops and transfinite sequence of functions (Proposition~\ref{prop: same superenvelops and sequence}). In particular, they still determine the same other entropy invariants.

We start by recalling the definitions introduced in \cite{downarowicz05}.

\begin{definition}
A non-decreasing sequence $\cH$ of functions $h_k : \cM(X,T) \to \bR$ converging pointwise to the metric entropy map is called a candidate (to $T$).
\end{definition}
An equivalence relation on candidates can be defined as follows. 
\begin{definition}
If $\cH = (h_k)_k$ and $\cH' = (h'_\ell)_\ell$ are candidates such that for any $\gamma > 0$ and $k \geqslant 0$ there is an $\ell \geqslant 0$ such that $h_k \leqslant h'_\ell + \gamma$, then we say that $\cH'$ uniformly dominates $\cH$, and we write $\cH' \unidom \cH$. \\
If $\cH' \unidom \cH$ and $\cH \unidom \cH'$, we say that $\cH$ and $\cH'$ are uniformly equivalent. 
\end{definition}
The notion of candidate is natural because the entropy map is often obtain as the pointwise limit of non-decreasing sequences of map, which are therefore candidates. In \cite{downarowicz05}, Downarowicz proved that several classical candidates are uniformly equivalent among each other (in the case $T$ is a homeomorphism).

The entropy structure is a distinguished equivalence class of candidates. Any candidate in the entropy structure equivalence class is also called an entropy structure.
\begin{definition}\label{def: entropy structure}
Given a continuous self-map $T : X \to X$, denote $T':X' \to X'$ any zero-dimensional principal extension of $T$, and $\cA_k$ a sequence of refining partitions of $X'$ into clopen sets. The reference candidate $\cH^{ref}$ is defined on $\cM(X',T')$ as the sequence of maps $ (\nu \mapsto h_\nu(T',\cA_k))_{k \in \bN}$. \\
A candidate to $T$ is called an entropy structure if its lift to $\cM(X',T')$ is uniformly equivalent to $\cH^{ref}$.
\end{definition}
One can prove that this definition depends neither on the choice of the zero-dimensional extension, nor on the choice of the refining sequence of clopen partitions of $X'$. Notice that two candidates are uniformly equivalent if and only if their lifts are. In particular, any two entropy structures are uniformly equivalent, and any candidate uniformly equivalent to an entropy structure is itself an entropy structure. 

The candidates Downarowicz proved equivalent are entropy structures.

We start with the following result.

\begin{lemma}\label{lemma: eq class refining pu}
For any sequences $\Psi_k$ and $\Psi'_k$ of continuous partitions of unity over $X$, $\Phi_n = \bigvee_{k=0}^n \Psi_k$ and $\Phi'_n = \bigvee_{k=0}^n \Psi'_k$, if
\begin{align*}
\lim_{n \to \infty} \tth_\mu(T,\Phi_n) = \lim_{n \to \infty} \tth_\mu(T,\Phi'_n) = h_{\mu}(T),
\end{align*}
for every $\mu \in \cM(X,T)$, then the candidates $\cH = (\mu \mapsto \tth_\mu(T,\Phi_n))_n$ and $\cH' = (\mu \mapsto \tth_\mu(T,\Phi_n))_n$ are uniformly equivalent. \\
We denote by $\cH^{p.u.}$ their common equivalence class.
\end{lemma}

\begin{proof}
In order to simplify notations, denote $h_n(\mu) =  \tth_\mu(T,\Phi_n)$ and similarly $h'_n(\mu) = \tth_\mu(T,\Phi'_n)$. First, notice that
\begin{align*}
h_{n+1}(\mu) = \tth_\mu(T,\Psi_{n+1} \vee \Phi_n) = \tth_\mu(T,\Phi_n) + \tth_\mu(T,\Psi_{n+1} | \Phi_n) \geqslant h_n(\mu).
\end{align*}
Thus $\cH$ is a candidate, as well as $\cH'$. The proof now is exactly as the one of \cite[Lemma~7.1.2]{downarowicz05}. For clarity, we repeat the argument. Note that, because of Theorem~\ref{thm: vanish diam metric},
\begin{align*}
\lim_{n \to \infty} \tth_\mu(T,\Phi'_k | \Phi_n) = \lim_{n \to \infty} \tth_\mu(T,\Phi'_k \vee \Phi_n) - \tth_\mu(T,\Phi_n) = h_\mu(T) - h_\mu(T) = 0.
\end{align*}
Hence $(\mu \mapsto\tth_\mu(T,\Phi'_k | \Phi_n))_n$ is a non-decreasing sequence of upper semi-continuous functions, defined over a compact set, converging point-wise to the continuous function constant equals to zero. Therefore the convergence is uniform. In particular, for any $k$ and $\ve > 0$, there exists $n$ such that uniformly in $\mu$,
\begin{align*}
h'_k(\mu) - h_n(\mu) \leqslant \tth_\mu(T,\Phi'_k \vee \Phi_n) - \tth_\mu(T,\Phi_n) \leqslant \ve.
\end{align*}
This is exactly $\cH \unidom \cH'$. By symmetry, exchanging the roles of $\cH$ and $\cH'$ yields $\cH' \unidom \cH$, which concludes the proof.
\end{proof}

\begin{proposition}\label{prop: increasing entropy structure with pu}
Any candidate in the class $\cH^{p.u.}$ is an entropy structure.
\end{proposition}

\begin{proof}
Let $(X',T')$ be a zero-dimensional principal extension of $(X,T)$, and $\pi : X' \to X$ the (continuous) topological factor map. Let $\cA_k$ be a refining sequence of partitions into clopen sets of $X'$. Denote $\cH^{ref} = (\nu \mapsto h_\nu(T',\cA_k))$ the reference entropy structure. Notice that since the atoms of each $\cA_k$ are clopen, the partitions of unity $\mathbbm{1}_{\cA_k}$ are continuous. We can see from Lemma~\ref{lemma: eq class refining pu}, applied to $(X',T')$ that a different choice of sequence $\cA_k$ leads to a candidate uniformly equivalent to $\cH^{ref}$. 

Consider now a sequence of continuous partitions of unity over $X$ of the form $\Phi_n = \bigvee_{k=0}^{n} \Psi_k$. Denote $\cH = (\mu \mapsto \tth_\mu(T,\Phi_n))$. Therefore, $\pi\Phi_n$ is a sequence of continuous partition of unity over $X'$. Furthermore, for any $\nu \in \cM(X',T')$,
\begin{align*}
\lim_{n \to \infty} \tth_\nu(T',\pi \Phi_n) = \lim_{n \to \infty} \tth_{\pi_* \nu}(T,\Phi_n) = h_{\pi_* \nu}(T) = h_\nu(T'),
\end{align*}
because the extension is principal. Therefore, the lift $\pi \cH$ of $\cH$ is a candidate for $(X',T')$ and $\pi\Phi_n = \bigvee_{k=0}^{n} \pi \Psi_k$. Since $h_\nu(T',\cA_k) = \tth_\nu(T', \mathbbm{1}_{\cA_k})$, by applying Lemma~\ref{lemma: eq class refining pu}, we get that $\pi \cH$ and $\cH^{ref}$ are uniformly equivalent. That is, $\cH$ is an entropy structure.
\end{proof}

Before considering continuous partitions of unity $\Phi_k$ with the only condition that $\diam(\Phi_k)$ converges to zero, we need to introduce another candidate -- the definition of candidate is slightly extended to allow for families indexed by $\delta > 0$, the convergence to the entropy map being understood as $\delta \to 0$.

\begin{definition}
For any $\delta > 0$, define
\begin{align*}
\tth_\mu(T,\delta) = \inf \{ \tth_\mu(T,\Phi) \mid \diam(\Phi) \leqslant \delta \},
\end{align*}
where the infimum ranges over continuous partition of unity.
\end{definition}

\begin{proposition}
The candidate $\cH^{loc} = (\mu \mapsto \tth_\mu(T,\delta))_{\delta > 0}$ is an entropy structure.
\end{proposition}

\begin{proof}
In order to simplify notation, denote $h_\delta(\mu) = \tth_\mu(T,\delta)$. We start by proving that $\cH^{loc}$ is indeed a candidate. It is clear from the definition that the family $h_\delta$ is monotone. We first check that $h_\delta$ converges point-wise to the metric entropy as $\delta$ goes to zero. Let $\Phi$ and $\Phi'$ be continuous partitions of unity. Therefore
\begin{align*}
\tth_\mu(T,\Phi) - \tth_\mu(T,\Phi') \leqslant \tth_\mu(T,\Phi \vee \Phi') - \tth_\mu(T,\Phi') = \tth_\mu(T, \Phi | \Phi') \leqslant \ttH_\mu(\Phi | \Phi') \leqslant \ttH(\Phi | \Phi').
\end{align*}
For any $\ve > 0$, there is a $\delta > 0$ depending only on $\ve$ and $\Phi$ such that if $\diam(\Phi') \leqslant \delta$, then $\ttH(\Phi | \Phi') < \ve$. Therefore 
\begin{align*}
\tth_\mu(T) \geqslant \limsup_{\delta \to 0} h_\delta(\mu) \geqslant \liminf_{\delta \to 0} h_\delta(\mu) \geqslant \tth_\mu(T,\Phi) - \ve.
\end{align*}
Taking the supremum over all $\Phi$ in the above, and the limit $\ve \to 0$ yield $\tth_\mu(T) = \lim_{\delta \to 0} h_\delta(\mu)$.

In order to prove the proposition, it is sufficient to prove that $\cH^{loc}$ is uniformly equivalent to a candidate from $\cH^{p.u.}$.

Let $\Phi_n = \bigvee_{k=0}^n \Psi_k$ where $\Psi_k$ are continuous partitions of unity with $\diam(\Psi_k)$ converging to zero. Denote $h_n(\mu) = \tth_\mu(T,\Phi_n)$ and $\cH = (h_n)_n$. Therefore, for any $\delta > 0$, there is a $k$ such that $\diam(\Psi_k) < \delta$. In particular, for every $n \geqslant k$, $\diam(\Phi_n) < \delta$ and therefore $h_\delta \leqslant h_n$. In other words, $\cH \unidom \cH ^{loc}$. 

Let $\Phi'$ be a continuous partition of unity. Therefore
\begin{align*}
h_n(\mu) - \tth_\mu(T,\Phi') \leqslant \tth_\mu(T,\Phi_n \vee \Phi') - \tth_\mu(T,\Phi') = \tth_\mu(T, \Phi_n | \Phi') \leqslant \ttH_\mu(\Phi_n | \Phi') \leqslant \ttH(\Phi_n | \Phi') .
\end{align*}
Now, for all $n$ and $\ve > 0$, there is a $\delta > 0$ such that if $\diam(\Phi') < \delta$, then $\ttH(\Phi_n | \Phi') < \ve$. Therefore $h_n \leqslant h_\delta + \ve$, or in other words $\cH^{loc} \unidom \cH$.
\end{proof}

\begin{corollary}\label{corollary: tail var pple}
For any sequence $\Psi_k$ of continuous partitions of unity with $\diam(\Psi_k)$ converging to $0$, it holds
\begin{align*}
h^*(T) = \lim_{k \to \infty} \sup_{\mu \in \cM(X,T)} h_\mu(T) - \tth_\mu(T,\Psi_k)
\end{align*}
\end{corollary}

\begin{proof}
Since $\cH^{p.u.}$ and $\cH^{loc}$ are entropy structures, it follows from \cite{downarowicz05} (in the case $T$ is invertible, and from \cite{burguet09} otherwise) that
\begin{align*}
h^*(T) 
= \lim_{\delta \to 0} \sup_{\mu \in \cM(X,T)} h_\mu(T) - \tth_\mu(T,\delta)
= \lim_{n \to \infty} \sup_{\mu \in \cM(X,T)} h_\mu(T) - \tth_\mu(T,\Phi_n),
\end{align*}
where $\Phi_n = \bigvee_{k=0}^n \Psi_k$. Now, taking $\delta_k = \diam(\Psi_k)$, it follows that for every $k \geqslant 0$,
\begin{align}\label{eq: comparison entropy structures}
\tth_{\mu}(T,\delta_k) \leqslant \tth_\mu(T,\Psi_k) \leqslant \tth_\mu(T,\Phi_k),
\end{align}
from which the claim follows easily.
\end{proof}

\subsection{Transfinite sequences and superenvelops}

To any candidate $\cH$, one can associate \emph{superenvelops} and a transfinite sequence of functions, indexed by ordinal and defined by transfinite induction. Before recalling the definitions, we introduce some notations. 

\begin{definition}
If $f : K \to \bR$ is a function, its upper semi-continuous envelop is defined as\footnote{The traditional notation is with a tilde. We change it here to avoid confusion with the other notation $\tth$.}
\begin{align*}
\overline{f}(x) \coloneqq \max\{ f(x) , \limsup_{y \to x} f(y) \}.
\end{align*}
The defect of upper semi-continuity of $f$ is defined as
\begin{align*}
\overset{...}{f} = \overline{f} - f.
\end{align*}
\end{definition}

\begin{definition}
Given a candidate $\cH = (h_k)$ and denoting its tails $\theta_k(\mu) = h_\mu(T) - h_k(\mu)$, we can define the transfinite sequence $u_\alpha^{\cH}$ of functions over $\cM(X,T)$ by 
\begin{align*}
u_0^{\cH} &\equiv 0, \\
u_\alpha^\cH &= \lim_{k \to \infty} \overline{ \sup_{\beta < \alpha} u_\beta^\cH + \theta_k }.
\end{align*}
A superenvelop of $\cH$ is any function $E : \cM(X,T) \to \bR$ such that $E(\mu) \geqslant h_\mu(T)$ and 
\begin{align*}
\lim_{k \to \infty} \overset{.............}{E-h_k} = 0.
\end{align*}
\end{definition}

Downarowicz proved that uniformly equivalent candidates shares the same transfinite sequence functions and the same superenvelops. In \cite{BD04}, Boyle and Downarowicz proved that to any candidate there is a countable ordinal $\alpha$ such that $u^{\cH}_{\beta} \equiv u^{\cH}_{\alpha}$ for every $\beta \geqslant \alpha$. The least ordinal $\alpha$ with this property is denoted $\alpha_0(\cH)$ and is called the order of accumulation of $\cH$. The importance of these notions comes from the case $\cH$ is an entropy structure: the pointwise infimum $E \cH$ over all superenvelops is a superenvelop which coincides with $\mu \mapsto h_\mu(T) + u^{\cH}_{\alpha_0(\cH)}(\mu)$, which itself coincides with the symbolic extension entropy.

\begin{proposition}\label{prop: same superenvelops and sequence}
If $\Psi_k$ is a sequence of continuous partitions of unity such that $\diam(\Psi_k)$ converges to $0$, then $\cH = (\mu \mapsto \tth_\mu(T,\Psi_k))_k$ has the same transfinite sequence of functions and superenvelops any entropy structure.
\end{proposition}

\begin{proof}
This is a direct consequence of \eqref{eq: comparison entropy structures}. More precisely, for $\alpha = 0$, we have $u_0 = u_0^{ref} \equiv 0$, where $u_\alpha^{ref}$ denotes the transfinite sequence of function associated to any entropy structure, and $u_\alpha$ the one associated to $\cH$. If $\alpha$ is an ordinal such that for every $\beta \leqslant \alpha$, $u_\beta = u_\beta^{ref}$, then set $\Phi_n = \bigvee_{k=0}^n \Psi_k$ and $\delta_k = \diam(\Psi_k)$, as well as the tails
\begin{align*}
\theta_k(\mu) &= h_\mu(T) - \tth_\mu(T,\Psi_k), \\
\theta_k^{p.u.}(\mu) &= h_\mu(T) - \tth_\mu(T,\Phi_k), \\
\theta_k^{loc}(\mu) &= h_\mu(T) - \tth_\mu(T,\delta_k).
\end{align*}
Therefore, since $\cH^{p.u.}$ and $\cH^{loc}$ are entropy structures, we get from \eqref{eq: comparison entropy structures} that
\begin{align*}
\limsup_{k \to \infty} \sup_{\beta < \alpha +1} \overline{u_\beta + \theta_k} \leqslant \lim_{k \to \infty} \sup_{\beta < \alpha +1} \overline{u^{ref}_\beta + \theta_k^{loc}} = u_{\alpha +1}^{ref},
\end{align*}
and
\begin{align*}
\liminf_{k \to \infty} \sup_{\beta < \alpha +1} \overline{u_\beta + \theta_k} \geqslant \lim_{k \to \infty} \sup_{\beta < \alpha +1} \overline{u^{ref}_\beta + \theta_k^{p.u.}} = u_{\alpha +1}^{ref}.
\end{align*}
In particular, $u_{\alpha+1}$ is well defined and coincides with $u_{\alpha +1}^{ref}$.

Let $E$ be a superenvelop of $\cH^{p.u.}$. Hence
\begin{align*}
0 \leqslant \lim_{k \to \infty} (\overset{...............}{E-h_k})(\mu) \leqslant \lim_{k \to \infty} (\overset{.................}{E-h_k^{p.u.}})(\mu) + \lim_{k \to \infty} (\overset{....................}{h_k^{p.u.}-h_k})(\mu) = 0,
\end{align*}
where we used that $(h_k^{p.u.}-h_k)(\mu) = \tth_\mu(T,\Phi_{k-1}|\Psi_k)$ is upper semi-continuous with respect to $\mu$.
Therefore $E$ is a superenvelop of $\cH$. Conversely, consider $E$ a superenvelop of $\cH$. Let $\mu \in \cM(X,T)$, $\ve > 0$ and $k$ large enough so that both $(\overset{...............}{E-h_k})(\mu) < \ve$ and $h_\mu(T) - h_k(\mu) < \ve$. There exists $n_k$ large enough so that both $h_k - h_{n_k}^{p.u.} < \ve$ and $h_\mu(T) - h_{n_k}^{p.u.}(\mu) < \ve$. Using that $h_k^{p.u.}$ is non-decreasing, we get
\begin{align*}
0 \leqslant \lim_{k \to \infty} (\overset{.................}{E-h_k^{p.u.}})(\mu) &\leqslant (\overset{.................}{E-h_{n_k}^{p.u.}})(\mu) \\
&\leqslant (\overset{...............}{E-h_k})(\mu) + (\overline{ h_k - h_{n_k}^{p.u.}})(\mu) - (h_k - h_{n_k}^{p.u.})(\mu) < 4\ve. 
\end{align*}
Therefore $E$ is also a superenvelop  of $\cH^{p.u.}$.
\end{proof}

\subsection{Larger equivalence classes}

For a general sequence of continuous partitions of unity $\Phi_k$ with $\diam(\Phi_k)$ converging to zero, the sequence of functions $\cH = (\mu \mapsto \tth_\mu(T,\Phi_k))_{k \in \bN}$ is not a candidate since we cannot ensure it is non-decreasing. By weakening the notions of candidate and of uniform equivalence, we can enlarge the notion of entropy structure without losing the uniformity of transfinite sequence of functions and superenvelops on a class. The sequence $\cH$ then belongs to the enlarged class of entropy structure. More precisely, we start by defining a suitable notion that replace the non-decreasing condition over candidates.

\begin{definition}
A sequence $(v_k)_k$ is said to be almost-increasing if for any $\gamma > 0$ and $k \geqslant 0$ there exists $\ell_{\gamma,k}$ such that for every $\ell \geqslant \ell_{\gamma,k}$, it holds $v_k \leqslant v_\ell + \gamma$. A sequence $(v_k)_k$ is said to be almost decreasing if $(-u_k)_k$ is almost increasing.
\end{definition}
This notion has the convenient following property.
\begin{lemma}
Any non-negative almost-decreasing sequence converges.
\end{lemma}

\begin{proof}
Let $v_k$ be an almost-decreasing sequence. Fix some $\ve > 0$ and $k$. There is an $n_{k,\ve}$ such that for every $n \geqslant n_{k,\ve}$, $v_n \leqslant v_k + \ve$. In particular, $ 0 \leqslant \limsup_{n \to \infty} v_n \leqslant v_k + \ve$. Since this holds for any $k$, we deduce that $\limsup_{n \to \infty} v_n \leqslant \liminf_{k \to \infty} v_k + \ve$, which yields the claim.
\end{proof}

We can now defined an enlarge notion of candidates and uniform equivalence between them.

\begin{definition}\label{def: weak}
We call weak-candidate any almost-increasing sequence $\cH$ of maps $h_k : \cM(X,T) \to \bR$ converging pointwise to the entropy map of $T$. \\
Given two weak-candidates $\cH = (h_k)_k$ and $\cH'= (h'_\ell)_\ell$, we say that $\cH'$ weakly-uniformly dominates $\cH$ if for any $\gamma > 0$ and any $k \geqslant 0$ there exists $\ell \geqslant 0$ such that $h_k \leqslant h'_\ell + \gamma$. We keep the notation $\cH' \unidom \cH$. \\
If $\cH' \unidom \cH$ and $\cH \unidom \cH'$, we say that $\cH$ and $\cH'$ are weakly-uniformly equivalent.
\end{definition}

\begin{theorem}\label{thm: same repair and superenvelops}
If $\cH$ and $\cH'$ are weakly-uniformly equivalent weak-candidates, then they have the same transfinite sequence of functions and the same superenvelops.
\end{theorem}

\begin{proof}
First, notice that the transfinite sequence of functions is well defined. Indeed, if $h_k$ is an almost-increasing sequence, then for any ordinal $\alpha$ such that $u_\beta$ is well defined for each $\beta < \alpha$, we get that for every $\mu \in \cM(X,T)$, the sequence $\overline{\sup_{\beta < \alpha} u_\beta + \theta_k}(\mu)$ is almost-decreasing and non-negative, hence converges. Thus $u_\alpha(\mu)$ is well defined.

Now, for $\cH = (h_k)$ and $\cH' = (h'_k)$, we get that for any fixed $\ve > 0$, and $k$ there is an $\ell' \geqslant k$ such that for every $\ell \geqslant \ell'$, $h_k \leqslant h'_\ell + \ve$. Therefore, for $\alpha = 0$ we have that $u_0 = u'_0$. Furthermore, if $u_\beta = u'_\beta$ for every $\beta < \alpha$, we get that
\begin{align*}
\overline{\sup_{\beta < \alpha} u_\beta + \theta_k } \geqslant \lim_{\ell \to \infty} \overline{\sup_{\beta < \alpha} u'_\beta + \theta'_\ell } - \ve.
\end{align*}
Letting $k$ tends to infinity and $\ve$ to zero, yield $u_\alpha \geqslant u'_\alpha$. By symmetry, exchanging the roles of $\cH$ and $\cH'$ gives the desired equality. 

Let $E$ be a superenvelop of $\cH$. Therefore, for $\ve > 0$ and $\mu \in \cM(X,T)$, we get
\begin{align*}
(\overset{..............}{E - h'_\ell})(\mu) \leqslant (\overset{..............}{E - h_k})(\mu) + (\overline{h_k - h'_\ell})(\mu) - (h_k - h'_\ell)(\mu).
\end{align*}
Chose $k$ large enough so that $(\overset{..............}{E - h_k})(\mu) < \ve$, and $h_\mu(T) - h_k(\mu) < \ve$. For any $\ell$ large enough so that $\ell > k$, $h_\mu(T) - h'_\ell(\mu) < \ve$ and $h_k - h'_\ell < \ve$, we get that
\begin{align*}
(\overset{..............}{E - h'_\ell})(\mu) < 4\ve.
\end{align*}
Since $(\overset{..............}{E - h'_\ell})(\mu)$ is a non-negative almost-decreasing sequence, it converges, and the limit has to be zero. Thus $E$ is also a superenvelop of $\cH'$.
\end{proof}

As a consequence, we have that,

\begin{corollary}\label{thm: tail var pple general version}
If $\cH = (h_k)_k$ is a weak-candidate weakly-uniformly equivalent to an entropy structure, then it satisfies the tail variational principle:
\begin{align*}
\lim_{k \to \infty} \sup_{\mu \in \cM(X,T)} h_\mu(T) - h_k(\mu) = h^*(T).
\end{align*}
\end{corollary}

In order to prove this corollary, we need a technical result allowing to exchange a limit and a supremum. We proceed in two steps.

\begin{lemma}
If $f_k$ is an almost-decreasing sequence of non-negative upper semi-continuous functions on a compact set $K$ converging pointwise to some function $f$, and $g$ is a continuous function such that $g > f$, then for any $\ve >0$ there is a $k_\ve$ such that for any $k \geqslant k_\ve$, $g + \ve > f_k$.
\end{lemma}

\begin{proof}
Let $\ve > 0$. From the sequence $(f_k)$ we can extract a subsequence $(f'_k)$ such that $f'_{k+1} \leqslant f'_k + 2^{-k} \ve$ for every $k \geqslant 0$. It remains that $f'_k$ converge pointwise to $f$. Define the sequence of sets
\begin{align*}
K_k \coloneqq \{ x \in K \mid (f'_k - g)(x) \geqslant \sum_{i=0}^{k-1} 2^{-i}\ve \}
\end{align*}
This is a decreasing sequence of compact sets. Indeed, since $f'_k-g$ is upper semi-continuous, the set $K_k$ is a closed subset of the compact set $K$, hence compact. Furthermore, if $x \in K_{k+1}$, then
\begin{align*}
(f'_k-g)(x) \geqslant (f'_{k+1}-g)(x) - 2^{k}\ve \geqslant \sum_{i=0}^{k} 2^{-i}\ve - 2^{k}\ve = \sum_{i=0}^{k-1} 2^{-i}\ve.
\end{align*}
Thus $x \in K_k$, and therefore $K_{k+1} \subset K_k$.

Now, for any $x \in K$, since $\lim_{k \to \infty} f'_k(x) = f(x) < g(x)$, we get that $x \notin K_k$ for any large enough $k$. Therefore $\bigcap_{k \geqslant 0} K_k = \emptyset$. Hence, there is a $k_\ve$ such that $K_k = \emptyset$ for any $k \geqslant k_\ve$. Since $(f_k)$ is almost decreasing, there is an $N$ such that for any $k \geqslant N$,
\begin{align*}
f_k \leqslant f'_{k_\ve} + \ve < g + 3 \ve.
\end{align*}
\end{proof}

\begin{proposition}
If $f_k$ is an almost-decreasing sequence of upper semi-continuous functions on a compact set $K$, then
\begin{align*}
\lim_{k \to \infty} \sup_{x \in K} f_k(x) = \sup_{x \in K} \lim_{k \to \infty} f_k(x).
\end{align*}
\end{proposition}

\begin{proof}
Since $f_k$ is almost decreasing, we get that for any $x \in K$, $\ell \geqslant 0$ and $\ve > 0$, $f_\ell(x) \geqslant \lim_{k \to \infty} f_k(x) - \ve$. Thus $\sup_{x \in K} f_\ell(x) \geqslant \sup_{x \in K} \lim_{k \to \infty} f_k(x)$. Therefore
\begin{align*}
\liminf_{k \to \infty} \sup_{x \in K} f_k(x) \geqslant \sup_{x \in K} \lim_{k \to \infty} f_k(x).
\end{align*} 
For the reverse inequality, first, let $g$ be the continuous function constant and equals to \[ g \equiv \limsup_{k \to \infty} \sup_{x \in K} f_k(x). \]
By contradiction, assume that $g - \sup_{x \in K} \lim_{k \to \infty} f_k(x) = \ve_0 > 0$. Denoting $f$ the pointwise limit of $f_k$, we get that $g \geqslant f + \ve_0 > f + \ve_0/2$. Therefore, there is a $N_0$ such that for any $\ell \geqslant N_0$, $f_\ell \leqslant g - \ve_0/4$.

Now, from the definition of $g$, we get that for every $k$ there is some $\ell \geqslant k$ such that $\sup_{x \in K} f_\ell(x) \geqslant g - \ve_0/8$. In particular, for $k = N_0$, we get that
\begin{align*}
g - \ve_0/8 \leqslant \sup_{x \in K} f_k(x) \leqslant g - \ve_0/4,
\end{align*}  
a contradiction. We have thus proved that
\begin{align*}
\limsup_{k \to \infty} \sup_{x \in K} f_k(x) \leqslant \sup_{x \in K} \lim_{k \to \infty} f_k(x) \leqslant \liminf_{k \to \infty} \sup_{x \in K} f_k(x),
\end{align*}
which concludes the proof.
\end{proof}

\begin{proof}[Proof of Corollary~\ref{thm: tail var pple general version}]
In \cite{burguet09}, Burguet proved that $\sup_{\mu \in \cM(X,T)} u_1(\mu) = h^*(T)$. Therefore
\begin{align*}
h^*(T) = \sup_{\mu \in \cM(X,T)} u_1(\mu) &= \sup_{\mu \in \cM(X,T)} \lim_{k \to \infty} (\overline{h_\cdot(T) - h_k})(\mu) \\
&= \lim_{k \to \infty} \sup_{\mu \in \cM(X,T)} (\overline{h_\cdot(T) - h_k})(\mu) \\
&= \lim_{k \to \infty} \sup_{\mu \in \cM(X,T)} h_\mu(T) - h_k(\mu),
\end{align*}
where for the last equality we used that a function and its upper semi-continuous have the same supremum.
\end{proof}

\begin{theorem}\label{thm: weak entropy structure}
If $\Psi_k$ a sequence of continuous partitions of unity such that $\diam(\Psi_k)$ converges to $0$, then $\cH = (\mu \mapsto \tth_\mu(T,\Psi_k))_k$ is a weak-candidate weakly-uniformly equivalent to any entropy structure.
\end{theorem}

\begin{proof}
Equation \eqref{eq: comparison entropy structures} yields $\cH^{p.u.} \unidom \cH \unidom \cH^{loc}$. Since $\cH^{p.u.}$ and $\cH^{loc}$ are both entropy structures, they are uniformly equivalent, and therefore weakly-uniformly equivalent.
\end{proof}

\section{Open questions}\label{sect: open}

In this section, we raise two questions, that are still open. First, in \cite{romagnoli03}, Romagnoli gave a definition of the local metric entropy function $\mu \mapsto h_\mu(T,\cU)$ involving over covers $\cU$. Since local topological pressure $\htop(T,\cU)$ with respect to the same open cover provides a bound from above, it is natural to ask whether the two quantities are related. Indeed, Romagnoli proved a local variational principle: $\sup_{\mu \in \cM(X,T)} h_\mu(T,\cU) = \htop(T,\cU)$ for any finite open cover $\cU$. This results has been extended to local pressures in~\cite{HY07} by Huang and Yi, to local conditional entropy in~\cite{R18} by Romagnoli and to local conditional pressure in~\cite{SL23} by Song and Li. 

In the context of the present article, we introduced local metric entropy $\mu \mapsto \tth_\mu(T,\Phi)$ and local topological entropy $\tthtop(T,\Phi)$ with respect to any finite continuous partition of unity $\Phi$, as well as conditional versions, and pressure. Furthermore, we proved the upper semi-continuity of the local (conditional) metric entropy function. In particular the supremum is achieved, but is it equal to the local topological entropy? That is

\begin{question}
For any continuous partition of unity $\Phi$, is it true that \[ \sup_{\mu \in \cM(X,T)} \tth_\mu(T,\Phi) = \tthtop(T,\Phi) \text{?}\]
\end{question}

We now turn to another question. In~\cite{DF05}, Downarowicz and Frej considered Markov operators (a generalisation of the Koopman operators). They proved that their is essentially one notion of metric entropy, as soon as a definition satisfies some reasonable axioms. They also introduced three definitions of topological entropy, and prove on the one hand that they provide the same quantity, and on the other hand that it is an upper bound for the metric entropy. The question of a variational principle relating the metric and topological notions is still open. 

The definition of metric and topological entropy provided in the present article can be extended to the case of Markov operators. The metric one satisfies the axioms provided by Downrowicz and Frej, and it is not too difficult to prove that the topological entropy here is less or equal to the topological entropy from \cite{DF05} (especially  their second definition $h_2(T)$). 

\begin{question}
Does the topological entropy of a Markov operator defined using continuous partitions of unity coincide with the other definitions provided in \cite{DF05}?
\end{question}

\bibliography{biblio.bib}{}
\bibliographystyle{abbrv}
%\nocite{*}

\end{document}